\documentclass[12pt,a4paper]{article}
\usepackage{float}
\usepackage{commath} 
\usepackage{enumerate}
\usepackage[left=1.2in,right=1.2in,top=1.4in,bottom=1.4in]{geometry}
\usepackage{subfig}

\usepackage{amsmath, amsthm, amssymb}
\usepackage{bm} 
\usepackage{mathtools} 

\usepackage[english]{babel}
 
\usepackage[colorlinks=true,linkcolor=blue,urlcolor=blue,citecolor=red]{hyperref}

\theoremstyle{plain}
\newtheorem{theorem}{Theorem}
\newtheorem{corollary}{Corollary}
\newtheorem{lemma}{Lemma}

\theoremstyle{definition}

\theoremstyle{remark}
\newtheorem{remark}{Remark}

\DeclarePairedDelimiter\norms{\Vert}{\Vert}

\DeclarePairedDelimiter\floor{\lfloor}{\rfloor}

\DeclareMathOperator{\as}{a.s.}

\renewcommand{\v}{\text{Var}}

\newcommand{\half}{\mbox{$\frac12$}}
\renewcommand{\Re}{\operatorname{Re}}

\DeclareMathOperator{\p  }{\normalfont{\textbf{P}}}
\DeclareMathOperator{\e  }{\normalfont{\textbf{E}}}
\DeclareMathOperator{\var}{\normalfont{\text{Var}}}

\newcommand{\bb}{B^{\circ}}
\newcommand{\C}{C_{n,k}}

\newcommand{\dto}{\Rightarrow}

\newcommand{\Xl}{X_n}

\newcommand{\Xbb}{\widetilde{X}_{n}}
\newcommand{\Xh}{\widehat{X}_n}
\newcommand{\scale}{w}

\newcommand{\driftspace}{b}

\newcommand{\driftX}{\mu_X} 
\newcommand{\driftXi}{b} 
\newcommand{\driftZeta}{\beta} 
\newcommand{\varianceXi}{a} 
\newcommand{\varianceZeta}{\alpha} 
\newcommand{\stddevZeta}{\alpha^{1/2}} 
\newcommand{\driftGauss}{\mu}
\newcommand{\stddevGauss}{\sigma}
\newcommand{\varianceGauss}{\sigma^2}
\newcommand{\lamperti}{\psi}

\newcommand{\dxi}{\Delta\xi}
\newcommand{\dtaylor}{Z_{n,k}}
\newcommand{\momentdxi}{\lambda}
\newcommand{\momentdtaylor}{\widetilde{\lambda}}

\newcommand{\maxTimeSpan}{\eta_2}
\newcommand{\minTimeSpan}{\eta_1}
\newcommand{\pbb}{\pi}
\newcommand{\FourierMoment}{\widehat{p}}

\newcommand{\Cf}{C}

\newcommand{\Df}{D}

\title{On Markov chain approximations for computing boundary crossing probabilities of diffusion processes\footnote{Research supported by the University of Melbourne Faculty of Science Research Grant Support Scheme.}}
\date{}
\author{Vincent Liang\footnote{School of Mathematics and Statistics, The University of Melbourne, Parkville 3010, Australia;  e-mail: vliang@student.unimelb.edu.au.}  {} and Konstantin Borovkov\footnote{School of Mathematics and Statistics, The University of Melbourne, Parkville 3010, Australia; e-mail: borovkov@unimelb.edu.au.}}

\begin{document}

\maketitle

\begin{abstract}
We propose a discrete time discrete space Markov chain approximation with a Brownian bridge correction for computing curvilinear boundary crossing probabilities of a general diffusion process on a finite time interval. For broad classes of curvilinear boundaries and diffusion processes, we prove the convergence of the constructed approximations in the form of products of the respective substochastic matrices to the boundary crossing probabilities for the process as the time grid used to construct the Markov chains is getting finer. Numerical results indicate that the convergence rate for the proposed approximation with the Brownian bridge correction is $O(n^{-2})$ in the case of $C^2$-boundaries and a uniform time grid with $n$ steps. 



\end{abstract}

\section{Introduction}

We consider the problem of approximating the probability for a general diffusion process to stay between two curvilinear boundaries. Mathematically, the problem is solved: the non-crossing probability can be expressed as a solution to the respective boundary value problem for the backward Kolmogorov partial differential equation (this result goes back to the 1930s, see~\cite{Kolmogorov1931}, \cite{Kolmogorov1933}, \cite{Khintchine1933}). However, simple explicit analytic expressions are confined to the case of the Wiener process using the method of images~\cite{lerche}, and most of the results for diffusion processes rely on verifying Cherkasov's conditions (see~\cite{Cherkasov1957}, \cite{Ricciardi1976}, \cite{Ricciardi1983}) and then transforming the problem to that for the Wiener process by using a monotone transformation. Outside this class of special cases, one should mostly rely on computing numerical approximations to the desired probabilities. 

One popular approach to finding expressions for the first-passage-time density is through the use of Volterra integral equations. Much work was done on the method of integral equations for approximating the first-passage time density for general diffusion processes (\cite{Ricciardi1983}, \cite{Ricciardi1984}, \cite{Buonocore1987}, \cite{Giorno1989}, \cite{Buonocore1990}, \cite{Jaimez1991}, \cite{Sacerdote1996}), culminating with \cite{Gutierrez1997}, which expressed the first-passage-time density of a general diffusion process in terms of the solution to a Volterra integral equation of the second kind. Volterra integral equations of the first kind for the first-passage time for the Brownian motion were derived in \cite{park76} and \cite{load}. Although the method of integral equations is quite efficient for computational purposes, a drawback of the method is that the kernel of the integral equation is expressed in terms of the transition probabilities of the diffusion process. Therefore, for this method to work, we have to first compute these transition probabilities for all times on the time grid used. The method proposed in this paper only requires knowledge of the drift and diffusion coefficients, allowing it to be easily used in the general case. For further details on the connection between Volterra integral equations and the first-passage-time density, we refer the reader to \cite{peskir2002}. Other computational techniques include Monte Carlo (\cite{Ichiba2011}, \cite{Gobet2000}) and numerical solving of partial differential equations (see e.g. \cite{Patie2008} and references therein).

The classical probabilistic approach to the boundary-crossing problem is the method of weak approximation, which involves proving that a sequence of simpler processes $X_n$ weakly converges to the desired diffusion process in a suitable functional space, which entails the convergence of the corresponding non-crossing probabilities. Along with the already mentioned \cite{Kolmogorov1931}, \cite{Kolmogorov1933}, \cite{Khintchine1933}, one can say that this approach was effectively used in \cite{erdos46} in the case of ``flat boundaries'' (see also Ch.\,2, \S\,11 in \cite{bill}). More recently, the authors in \cite{fuwu} approximated the Wiener process with a sequence of discrete Markov chains with absorbing states and expressed the non-crossing probability as a product of transition matrices. The authors in \cite{Ji2015} extended the results in \cite{fuwu} by approximating a general diffusion with a sequence of Markov chains and similarly expressed the non-crossing probability as a product of transition matrices. However, the convergence rates of these approximations are proved to be $O(n^{-1/2})$, which leaves much to be desired in practical applications.

Another standard approach is to approximate the true boundary with one for which the crossing probability is easier to compute. In the one-sided boundary case, the authors of \cite{wang97}, \cite{potz01} used piecewise-linear approximations and the well-known formula for a one-sided linear boundary crossing probability for the Brownian bridge process to express the non-crossing probability in terms of a multiple Gaussian integral. This was generalised to the case of two-sided boundaries in \cite{nov99}. In the case of the Brownian motion, a sharp explicit error bound for the approximation of the corresponding boundary crossing probabilities was obtained in \cite{borov} and extended to general diffusion processes in \cite{Downes2008}.

The present paper combines piecewise-linear boundary approximations, limit theorems on convergence of Markov chains to diffusions, and a modification of the matrix multiplication scheme from \cite{fuwu} to create an efficient and tractable numerical method for computing the boundary crossing probabilities for time-inhomogeneous diffusion processes in both one- and two-sided boundary cases. The approach in the paper consists of the following steps: 
    \begin{enumerate}[(i)]
        \item transform the original general diffusion process into one with unit diffusion coefficient, applying the same transformation to the boundaries;
        \item approximate the transformed diffusion process with a discrete time Gaussian Markov process using a weak Taylor expansion;
        \item approximate the discrete time process from step (ii) with a discrete Markov chain in discrete time, whose transition probabilities are given by the normalised values of the transition density of the process from step (ii). The state spaces of the discrete Markov chain are constructed in such a way that the Markov chain does not overshoot or undershoot the boundaries when hitting them;
        \item construct a continuous time process that interpolates the Markov chain from step~(iii) with a collection of Brownian bridges; and finally
        \item approximate the transformed boundaries with piecewise linear ones and compute the piecewise linear boundary crossing probability of the interpolated process constructed in step (iv) using matrix multiplication.
    \end{enumerate}

The paper is structured as follows. In Section 2 we describe in detail the above steps. Section 3 states the main results of the paper. In Section 4 we present the proofs of these results. Section 5 contains numerical examples.

\section{Markov chain approximation of a diffusion}

Step (i). We are interested in the boundary crossing probability of a one dimensional diffusion process $Y$, whose dynamics are governed by the following stochastic differential equation:
    \begin{equation}\label{sde}
        \begin{cases}
            dY(t) = \mu_Y(t,Y(t))\,dt+ \sigma_Y(t,Y(t))\,dW(t), & t\in (0,1],\\
            Y(0) = y_0,
        \end{cases}
    \end{equation}
where $\{W_t\}_{t\geq 0}$ is a standard Brownian motion process defined on a stochastic basis $(\Omega,\mathcal{F},(\mathcal{F}_{t})_{t\geq 0},\p)$ and $y_0$ is constant. We assume that $\mu_Y:[0,1] \times \mathbb{R} \to \mathbb{R}$ and $\sigma_Y : [0,1] \times \mathbb{R} \to (0,\infty) $ satisfy the following conditions sufficient  for the uniqueness and existence  of a strong solution to \eqref{sde} (see pp.\,297--298 in \cite{ethier09}):
    
\begin{enumerate}
    \item [(C1)] The functions $\mu_Y$ and $\sigma_Y$ are continuous and such that, for some $K < \infty$,
        \begin{equation*}
            x \mu_Y(t,x) \leq K(1 +x^2),\quad \sigma_Y^2(t,x) \leq K(1 + x^2)
        \end{equation*}
    for all $(t,x)\in [0,1]\times \mathbb{R}$ and, for any bounded open set $U \subset \mathbb{R}$ there exists a $K_U < \infty$ such that
        \begin{equation*}
            \lvert \mu_Y(t,x) - \mu_Y(t,y)\rvert + \lvert \sigma_Y^2(t,x) - \sigma_Y^2(t,y)\rvert \leq K_U\lvert x- y\rvert
        \end{equation*}
    for all $t\in [0,1]$, $x,y\in U$.
\end{enumerate}    
For the unit-diffusion transformation $\lamperti_t(y):= \int_0^y\sigma_Y(t,u)^{-1}\,du$ to be well-defined, we will also assume that
\begin{enumerate}
    \item [(C2)] the function $\sigma_Y$ is continuously differentiable with respect to $(t,x)$ with bounded partial derivatives and 
        \begin{equation*}
            \inf_{(t,x) \in [0,1]\times  \mathbb{R}}\sigma_Y(t,x)>0.
        \end{equation*}
\end{enumerate}

By It\^o's lemma, the transformed process $X =\{X(t) := \lamperti_t(Y(t))\}_{t\in[0,1]}$ is a diffusion process with a unit diffusion coefficient (see, e.g.,~Ch.\,4, \S\,4.7 of \cite{Kloeden1992} or \cite{Rogers1985}), 
    \begin{equation}\label{defn:sde_X}
        \begin{cases}
            dX(t) = \driftX(t,X(t))\,dt + dW(t),&t\in (0,1],\\
            X(0) = x_0 := \lamperti_0(y_0),
        \end{cases}
    \end{equation}
where $\driftX(t,x) = (\partial_t \lamperti_t + \frac{\mu_Y}{\sigma_Y}  - \half \partial_x\sigma_Y)\circ \lamperti_t^{-1}(x)$ and $ \lamperti_t^{-1}(z)$ is the inverse of $z= \lamperti_t(y)$ in $y$. 

Denote by $\Cf=C([0,1])$ the space of continuous functions $x : [0,1]\to\mathbb{R}$ equipped with the uniform norm $\norms{x}_{\infty} := \sup_{t\in[0,1]}\lvert x(t)\rvert$. For a fixed $x_0 \in \mathbb{R}$, consider the class 
    \begin{equation*}
        \mathcal{G}:= \{ (f^-,f^+): f^{\pm}\in \Cf,\,f^-(0) < x_0 < f^+(0),\, \inf_{0\leq t \leq 1}(f^+(t)-f^-(t)) > 0 \}
    \end{equation*}
of pairs of functions from $\Cf$ and introduce the notation
    \begin{equation*}
        S(f^-,f^+):= \{x \in \Cf:f^-(t) < x(t) <f^+(t),\,t\in [0,1] \},\quad (f^-,f^+) \in \mathcal{G}.
    \end{equation*}
The problem we deal with in this paper is how to compute the probability of the form $\p(Y \in S(g_0^-,g_0^+))$ for $(g_0^-,g_0^+) \in \mathcal{G}$. Clearly, the desired probability coincides with 
    \begin{equation*}
        \p(X \in G),\quad G:= S(g^-,g^+),
    \end{equation*}
where $g^{\pm}(t) := \lamperti_t(g_0^{\pm}(t))$, $t\in[0,1]$. It is also clear that $(g^{-},g^{+}) \in \mathcal{G}$ due to condition (C2). Henceforth, we work exclusively with the process $X$ and the boundaries $g^{\pm}$.

Step (ii). In the context of curvilinear boundary crossing probabilities, the authors in \cite{fuwu} approximated the Wiener process by discrete Markov chains with transition probabilities computed by first taking the values of the Wiener process transition densities on a lattice and then normalising these to obtain a probability distribution on that lattice (there is also a small adjustment of that distribution to perfectly match the first two moments of the original transition probabilities and the ``discretised'' ones). In the general case, due to the absence of closed-form expressions for the transition density of $X$, we use transition probabilities of the weak Taylor approximations of $X$ to construct the approximating discrete process.

More precisely, for any $n \geq 1$, let $0 = t_{n,0} < t_{n,1} <\cdots < t_{n,n} = 1$ be a partition of $[0,1]$, and set $\Delta t_{n,k}:= t_{n,k} - t_{n,k-1}$, $k=1,\ldots,n$. For simplicity, we assume that there exist constants $\minTimeSpan>0$ and $\maxTimeSpan \geq 1$ such that 
    \begin{equation}\label{eqn:d2}
        \frac{\minTimeSpan}{n}\leq \Delta t_{n,k}\leq \frac{\maxTimeSpan}{n},\quad k=1,2,\ldots,n,
    \end{equation}
for all $n \geq 1$. We will use the second order expansion, so we further require that the following condition is met:
    \begin{enumerate}
    \item [(C3)] For any fixed $x \in \mathbb{R}$, one has $\driftX(\,\cdot\,,x) \in C^1([0,1])$, and for any fixed $t \in [0,1]$, one has $\driftX(t,\,\cdot\,) \in C^2(\mathbb{R})$. Moreover, for any $r>0$, there exists a $K_r <\infty$ such that one has
        \begin{equation*}
            \lvert \driftX(t,x)\rvert + \lvert \partial_t \driftX(t,x)\rvert + \lvert \partial_{x} \driftX(t,x) \rvert + \lvert \partial_{xx} \driftX(t,x) \rvert \leq K_r , \quad t\in[0,1], \,\lvert x \rvert \leq r.
        \end{equation*}
\end{enumerate} 
For $n\geq 1$ and $k=1,\ldots,n,$ introduce a discrete scheme drift $\driftZeta_{n,k}$ and diffusion $\varianceZeta_{n,k}^{1/2}$ coefficients by
            \begin{equation}\label{defn:b}
                \driftZeta_{n,k}(x) := 
                    (\driftX + \half\Delta t_{n,k}(\partial_t \driftX+\driftX\partial_x\driftX + \half\partial_{xx}\driftX))(t_{n,k-1},x)
            \end{equation}
and
 \begin{equation}\label{defn:sigma}
                \varianceZeta_{n,k}^{1/2}(x) :=  
                    1 + \half\Delta t_{n,k} \partial_x \driftX  (t_{n,k-1},x).
            \end{equation}
For each fixed $n \geq 1$, the $n$th second-order weak Taylor approximation of the diffusion \eqref{defn:sde_X} is defined as the discrete time process $\zeta_n :=\{\zeta_{n,k}\}_{k=0}^n$, specified by $\zeta_{n,0} = x_0$ and 
    \begin{equation*}
            \zeta_{n,k}= \zeta_{n,k-1} + \driftZeta_{n,k}(\zeta_{n,k-1})\Delta t_{n,k} + \alpha_{n,k}^{1/2}(\zeta_{n,k-1})(\Delta t_{n,k})^{1/2}Z_{n,k},\quad k=1,\ldots,n,
    \end{equation*}
where $\{Z_{n,k}\}_{k=1}^n$ is a triangular array of independent standard normal random variables.  For more detail on weak Taylor approximations to solutions of stochastic differential equations, see Ch.\,14 in \cite{Kloeden1992}. Clearly, the conditional distributions of the increments $\zeta_{n,k}- \zeta_{n,k-1}$ given $\zeta_{n,k-1}$ are Gaussian, and so the transition probabilities of the discrete time process $\zeta_{n}$ can be easily obtained. 

\begin{remark}
Note that directly using a weak Taylor expansion for a general diffusion process with a space-dependent diffusion coefficient results in a non-central $\chi^2$ distribution for the transition probabilities (see e.g. \cite{Elerian1998}), which artificially limits the domain of the approximating process. For other approaches to approximating the transition density of a general diffusion process, see e.g. \cite{Hurn2007}, \cite{Li2013} and the references therein.
\end{remark}

Step (iii). Next we further approximate the discrete time continuous state space process $\zeta_n$ with a discrete time discrete state space Markov chain $\xi_n:=\{\xi_{n,k}\}_{k=0}^n$, whose transition probabilities are based on the normalised values of the transition density of $\zeta_n$. 

To improve the convergence rate for our approximations to $\p(X \in G)$, we construct our Markov chains $\xi_n$ choosing, generally speaking, different state spaces for each time step. Namely, the state space $E_{n,k}$ for $\xi_{n,k}$, $k=0,1,\ldots,n$, is chosen to be a lattice such that $g^{\pm}(t_{n,k}) \in E_{n,k}$. This modification improves upon the Markov chain approximation suggested in \cite{fuwu} and is widely used for accelerating the convergence rate of numerical schemes in barrier option pricing (see e.g. \cite{cheuk96}). More precisely, the spaces $E_{n,k}$ are constructed as follows. Let $g_{n,k}^{\pm} := g^{\pm}(t_{n,k})$, $k=1,\ldots,n,$ and, for fixed $\delta\in (0,\half]$ and $\gamma>0$, set
    \begin{equation*}
        \scale_{n,k}   := 
            \begin{cases}
                \frac{(g_{n,k}^+ -g_{n,k}^-)/(\Delta t_{n,k})^{1/2 +\delta} }{\floor{\gamma(g_{n,k}^+ -g_{n,k}^-)/(\Delta t_{n,k})^{1/2 +\delta}}},& 1 \leq k<n,\\
                \frac{(g^+(1) -g^-(1))/\Delta t_{n,n} }{\floor{\gamma(g^+(1) -g^-(1))/\Delta t_{n,n}}},& k=n,
            \end{cases}
    \end{equation*}
assuming that $n$ is large enough such that the integer parts in all the denominators are non-zero. We set the time-dependent space lattice step sizes to be
    \begin{equation}\label{eqn:h}
        h_{n,k} := 
        \begin{cases}
            \scale_{n,k}(\Delta t_{n,k})^{1/2 +\delta},& 1 \leq k < n, \\
            \scale_{n,n}\Delta t_{n,n}, & k =n.
        \end{cases}
    \end{equation}

\begin{remark}
The choice of the lattice spans in \eqref{eqn:h} can be explained as follows. Our approximation scheme involves replacing the original boundaries with piecewise linear ones (in Step (v)), which introduces an error of the order $O(n^{-2})$ (under  assumption \eqref{eqn:d2} and given that the functions $g^{\pm}$ are twice continuously differentiable). Therefore, there is no point in using $h_{n,k}$ smaller than necessary to achieve the above precision. Heuristic arguments (to be published elsewhere) show that it suffices to choose $h_{n,n}\asymp n^{-1}$, whereas $h_{n,k} \asymp n^{-\frac{1}{2}-\delta},$ $k <n,$ can be much larger. This is so because our Markov chain computational algorithm can be viewed as repeated trapezoidal quadrature and the partial derivative in $y$ of the taboo (on boundary crossing) joint density of $(X(t_{n,k-1}),X(t_{n,k+1}))$ given $X(t_{n,k})=y$ at the boundary $y =g(t_{n,k})$ at all time steps before the terminal time is zero, ensuring a higher approximation rate at these steps compared to the terminal one (cf. the Euler--Maclaurin formula)
\end{remark}

\begin{remark}
If we used $\delta =0$ in \eqref{eqn:h}, there would be no convergence of the sequence of processes $\{\xi_n\}_{n \geq 1}$ to the desired diffusion limit. This is because there would no convergence of the moments of the increments (cf. Lemma \ref{lemma:moment_gauss_sde}). Note also that, Instead of using the power function \eqref{eqn:h}, we could take $h_{n,k}:= w_{n,k}(\Delta t_{n,k})^{1/2}\psi(\Delta t_{nk,})$ for some $\psi(x) \to 0$ as $x \downarrow 0$, with $w_{n,k} := v_{n,k}/\floor{v_{n,k}}$ where
    \begin{equation*}
        v_{n,k} := \frac{g_{n,k}^+ -g_{n,k}^-}{(\Delta t_{n,k})^{1/2}\psi(\Delta t_{n,k})}
    \end{equation*}
(in our case, $\psi(x) = x^{\delta}/\gamma$). We chose the power function for simplicity's sake.
\end{remark}

The state space for $\xi_{n,k}$ is the $h_{n,k}$-spaced lattice  
    \begin{equation*}
        E_{n,k}:= \{ g_{n,k}^+ - j h_{n,k} : j \in \mathbb{Z}\},\quad k = 1,\ldots,n.
    \end{equation*}
We also put $E_{n,0}:=\{x_0\}$ and
    \begin{equation*}
        E_{n}:= E_{n,0}\times E_{n,1} \times \cdots \times E_{n,n}.
    \end{equation*}
Note that $\max_{1\leq k\leq n}\lvert  \scale_{n,k} - \gamma^{-1}\rvert \to 0$ as $n\to\infty$ and $\gamma^{-1}\leq \scale_{n,k}\leq 2\gamma^{-1}$ for all $1\leq k\leq n$ (assuming, as above, that $n$ is large enough).

Further, for $k=1,\ldots,n,$ we let
    \begin{equation}\label{defn:mu_and_sigma}
        \driftGauss_{n,k}(x) := \driftZeta_{n,k}(x)\Delta t_{n,k},\quad \varianceGauss_{n,k}(x):= \varianceZeta_{n,k}(x)\Delta t_{n,k},
    \end{equation}
and define $\xi_{n,k}$, $k=1,\ldots,n,$ $n\geq 1,$ as a triangular array of random variables, where each row forms a Markov chain with one-step transition probabilities given by
    \begin{equation*}
       p_{n,k}(x,y) := 
        \p(\xi_{n,k}=y\,|\,\xi_{n,k-1}=x)= \varphi(y \,|\,  x + \driftGauss_{n,k}(x),\varianceGauss_{n,k}(x))\frac{h_{n,k}}{\C(x)}
    \end{equation*} 
for $(x,y) \in E_{n,k-1}\times E_{n,k}$, where $\varphi(x\,|\,\mu,\sigma^2):= (2\pi \sigma^2)^{-1/2}e^{-(x -\mu)^2/(2\sigma^2) }$, $x\in\mathbb R$, denotes the Gaussian density with mean $\mu$ and variance $\sigma^2$ and
\begin{equation}\label{defn:scalingFactor}
        \C(x) := \sum_{y \in E_{n,k}}\varphi(y \,|\, x + \driftGauss_{n,k}(x),\varianceGauss_{n,k}(x))h_{n,k}, \quad x \in \mathbb{R}.
    \end{equation}
\begin{remark}
Choosing suitable $\delta < \frac{1}{2}$ and $\gamma>0$ in the definition of $h_{n,k}$ for $k <n$ can be used to reduce computational burden. In our numerical experiments, we found that reducing the value of $\delta \in (0,\frac{1}{2})$ did not negatively affect the empirical convergence rates for boundary crossing probabilities if $\gamma$ is sufficiently large ($\gamma \geq 1.5$). In the case $\delta =0$, our proposed scheme no longer converges since the infinitesimal moments do not converge (cf. Lemma \ref{lemma:moment_gauss_sde}). To restore convergence, one could use the adjusted transition probabilities suggested in \cite{fuwu}. However, we would not be able to apply the shifting state space methodology above since the approximation in \cite{fuwu} relies on adjusting the transition probabilities on a state space that is not changing each time step. 
\end{remark}

Step (iv). As an initial approximation for $\p(X \in G )$, one could take 
    \begin{equation*}
        \p(g_{n,k}^-< X(t_{n,k}) < g_{n,k}^+, k=1,\ldots,n),
    \end{equation*}
which can in turn be approximated using the Chapman--Kolmogorov equations by
    \begin{equation*}
        \p(g_{n,k}^- < \xi_{n,k} < g_{n,k}^+, \, k = 1,\ldots,n) 
        = \sum_{\bm{x}\in E_n^G} \prod_{k=1}^n p_{n,k}(x_{k-1},x_k),
    \end{equation*}
where $\bm{x} = (x_0,x_1,\ldots,x_n)$ and $E_{n}^G:= E_{n,0}\times E_{n,1}^G \times \cdots \times E_{n,n}^G$,
\begin{equation*}
    E_{n,k}^G:= \{x \in E_{n,k} : g^-_{n,k} < x < g^+_{n,k}\},\quad k = 1,\ldots,n.
\end{equation*}
Without loss of generality, one can assume that $n$ is large enough so that none of $E_{n,k}^G$ is empty.

An approximation of this kind was used in \cite{fuwu} to approximate the boundary crossing probability of the Brownian motion. Such discrete time approximations of boundary crossing probabilities are well-known to converge slowly to the respective probability in the ``continuously monitored'' case since they fail to account for the probability of crossing the boundary by the continuous time process at an epoch inside a time interval between consecutive points on the time grid used. 

In order to correct for this, so-called ``continuity corrections'' have been studied in the context of sequential analysis \cite{Siegmund1982} and, more recently, in the context of barrier option pricing \cite{Broadie1997}. These types of corrections have also been used in \cite{potz01} to correct for discrete time monitoring bias in Monte Carlo estimates of the boundary crossing probabilities of the Brownian motion. Without such a correction, using the classical result from \cite{nagaev}, the convergence rate of the approximation from \cite{fuwu} was shown to be $O(n^{-1/2})$. However, as our numerical experiments demonstrate, using our more accurate approximations of the transition probabilities in conjunction with the Brownian bridge correction greatly improves it from $O(n^{-1/2})$ to $O(n^{-2})$. 

In the case of the standard Brownian motion, the correction consists of simply multiplying the one-step transition probabilities by the non-crossing probability of a suitably pinned Brownian bridge. Due to the local Brownian nature of a diffusion process, it was shown in~\cite{baldi02} that the leading order term of the diffusion bridge crossing probability is given by an expression close to that of the Brownian bridge. Thus, to account for the possibility of the process $X$ crossing the boundary inside time intervals $[t_{n,k-1},t_{n,k}]$, we define a process $\Xbb := \{\Xbb(t)\}_{t\in[0,1]}$ which interpolates between the subsequent nodes $(t_{n,k}, \xi_{n,k})$ with a collection of independent Brownian bridges. For $k=1,\ldots,n,$ we set
\begin{equation}\label{defn:xbb}
        \Xbb(t) := \xi_{n,k-1}  + (\xi_{n,k} - \xi_{n,k-1})\frac{t-t_{n,k-1} }{\Delta t_{n,k}}+B_{n,k}^{\circ}(t),\quad t \in [t_{n,k-1},t_{n,k}],
    \end{equation}
where $B_{n,k}^{\circ}$ are independent Brownian motions ``pinned'' at the time-space points $(t_{n,k-1},0)$ and $(t_{n,k},0)$, these Brownian bridges being independent of $\xi_n$. Analogous to Theorem 1 in \cite{nov99}, using the Chapman--Kolmogorov equations, the non-crossing probability of the boundaries $g^{\pm}$ for $\Xbb$ can be expressed as
    \begin{equation}\label{eqn:Xbb_bcp}
        \p(\Xbb \in G)= \e\prod_{k=1}^{n}\left(1-\pbb_{n,k}(g^-,g^+\,|\, \xi_{n,k-1},\xi_{n,k})\right),
    \end{equation}
where $1-\pbb_{n,k}(g^-,g^+\, | \, x,y)$ is the probability that the trajectory of a Brownian motion pinned at the points $(t_{n,k-1},x)$ and $(t_{n,k},y)$ stays between the boundaries $g^{\pm}(t)$ during the time interval $[t_{n,k-1},t_{n,k}]$. 

Step (v). The above representation can be equivalently written as a matrix product,
\begin{equation}\label{eqn:bcp_gauss}
        \p( \Xbb \in G ) = \mathbf{T}_{n,1}\mathbf{T}_{n,2}\cdots \mathbf{T}_{n,n} \mathbf{1}^\top,
    \end{equation}
where $\bm{1}=(1,\ldots,1)$ is a row vector of length $\lvert E_{n,n}^G \rvert$, and the sub-stochastic matrices $\mathbf{T}_{n,k}$ of dimensions $\lvert E_{n,k-1}^G \rvert \times \lvert E_{n,k}^G \rvert$ have entries equal to the respective taboo transition probabilities
    \begin{equation*}
            (1-\pbb_{n,k}(g^-,g^+ \,|\, x,y))p_{n,k}(x,y), \quad  (x,y) \in E_{n,k-1}^G \times E_{n,k}^G.
    \end{equation*}
Unfortunately, closed-form expressions for curvilinear boundary Brownian bridge crossing probabilities $\pbb_{n,k}(g^-,g^+\,|\, x,y)$ are known in a few special cases only, so we approximate the original boundaries $g^{\pm}$ with piecewise linear functions $f_n^{\pm}$, which linearly interpolate between the subsequent nodes $(t_{n,k},g_{n,k}^{\pm})$ for $k=0,\ldots,n.$ In the special case of a one-sided boundary (when $g^-=-\infty$), the expression for the linear boundary crossing probability of the Brownian bridge is well-known (see, e.g.,~p.\,63 of \cite{borodin}):
    \begin{equation*}
        \pbb_{n,k}(-\infty ,f_n^+ \,|\, x,y) =  \mbox{$\exp\{ \frac{-2}{\Delta t_{n,k}}(g_{n,k-1}^+ - x)(g_{n,k}^+ -y) \}$},\quad x<g_{n,k-1}^+,y<g_{n,k}^+.
    \end{equation*}
In the case when both the upper and lower boundaries $f_n^{\pm}$ are linear, if the time interval $\Delta t_{n,k}$ is sufficiently small, we can approximate the Brownian bridge crossing probability with the sum of one-sided crossing probabilities:
    \begin{equation*}
        \pbb_{n,k}(f_n^-,f_n^+\,|\, x,y) = \pbb_{n,k}(f_{n}^-,\infty \,|\, x,y) + \pbb_{n,k}(-\infty,f_{n}^+ \,|\, x,y) - \vartheta(x,y,\Delta t_{n,k}),
    \end{equation*}
where the positive error term $\vartheta$ admits the obvious upper bound
    \begin{align*}
        &\vartheta(x,y,\Delta t_{n,k})\\
        &\leq \pi_{n,k}(f_n^-,\infty\,|\,x,y)\max_{t\in[t_{n,k-1},t_{n,k}]}P_{t,f_n^-(t);t_{n,k},y}\biggl(\sup_{s\in [t,t_{n,k}]}(W(s)-f_n^+(s))\geq 0\biggr)\\
        &\quad + \pi_{n,k}(-\infty,f_n^+\,|\,x,y)\max_{t\in[t_{n,k-1},t_{n,k}]}P_{t,f_n^+(t);t_{n,k},y}\biggl(\inf_{s\in [t,t_{n,k}]}(W(s)-f_n^-(s))\leq 0\biggr),
    \end{align*}
where $P_{s,a;t,b}(\cdot) := \p(\,\cdot\,|\, W(s)=a,W(t)=b)$. An infinite series expression for $\pbb_{n,k}(f_n^-,f_n^+\,|\, x,y)$ can be found e.g.~in \cite{hall}. 

We can further apply our method to approximate probabilities of the form $\p(X \in G, X(1) \in [a,b])$, for some $[a,b] \subseteq [g^-(1), g^+(1)]$ as well. We first replace the final space grid $E_{n,n}^G$ with
    \begin{equation*}
        E_{n,n}^{[a,b]} := \Big\{a \leq x \leq b : x = b - j\frac{b-a}{\floor{\gamma (b-a)/\Delta t_{n,n}}} , j\in \mathbb{Z}  \Big\}.
    \end{equation*}
Then, instead of \eqref{eqn:bcp_gauss} we have
    \begin{equation*}
        \p(\Xbb \in G,\Xbb(1) \in [a,b]) = \mathbf{T}_{n,1}\mathbf{T}_{n,2} \cdots \mathbf{T}_{n,n}\bm{1}_{[a,b]}^{\top},
    \end{equation*}
where $\mathbf{T}_{n,n}$ is now of dimension $\lvert E_{n,n-1}^G\rvert \times \lvert E_{n,n}^{[a,b]}\rvert$ and $\bm{1}_{[a,b]} = (1,\ldots,1)$ is a row vector of length $\lvert E_{n,n}^{[a,b]}\rvert$.

\begin{remark}
As a function of the ``forward variables'' $(t,y)$, the taboo transition density of the diffusion process $Y$ 
    \begin{equation*}
        u(t,y) := \mbox{$\frac{d}{dy}$}\p(Y(t) \leq y, g^-(r) < Y(r) < g^+(r),\, r\in[0,t]\,|\, Y(0) = y_0)
    \end{equation*}
is the unique solution to the forward Kolmogorov equation 
    \begin{equation*}
        \half\partial_{yy}(\sigma^2(t,y)u(t,y)) -  \partial_y(b(t,y) u(t,y)) -\partial_tu(t,y) =0
    \end{equation*}
in the domain $\{(t,y):t\in (0,1), y \in (g^-(t), g^+(t))\}$ with boundary conditions
    \begin{equation*}
        \begin{cases}
            \lim_{t\downarrow 0}u(t,y) = \delta_{0}(y-y_0),& y \in (g^-(0),g^+(0)),\\
            u(t,g^{\pm}(t)) = 0, & t\in[0,1],
        \end{cases}
    \end{equation*}
where $\delta_0$ is the Dirac delta function. For a proof, see, e.g.,~\cite{lerche}.

Commonly used numerical techniques for approximating the solution to the forward Kolmogorov equation include the method of finite differences \cite{Smith1986}, the finite element method \cite{Johnson1988} and the method of lines \cite{Hundsdorfer2003}. The connection between finite differences and the method of Markov chain approximations is explored in detail in \cite{Kushner1976a} and \cite{Platen2010}. 
\end{remark}
    
\begin{remark}
In the standard Brownian motion case, the matrix multiplication scheme in \eqref{eqn:bcp_gauss} can be seen as recursive numerical integration using the trapezoidal rule (cf. Remark 3 in \cite{nov99}). To construct a Markov chain approximation based on numerical integration techniques that use variable node positioning (e.g., the Gauss--Legendre quadrature), it would require interpolation of the resulting evolved transition density, causing the matrix multiplication in \eqref{eqn:Xbb_bcp} to lose its probabilistic meaning. Efficient numerical integration techniques based on analytic mappings (e.g.~the double--exponential quadrature \cite{Takahasi1973}) that maintain the positions of nodes are numerically feasible, however verifying theoretical weak convergence for the resulting sequence of Markov chains is more difficult. Furthermore, using the trapezoidal rule allows one to use the Euler--Maclaurin summation formula to improve the convergence rate. From our numerical experimentation, the one-dimensional trapezoidal rule that we propose in this paper strikes a good balance between simplicity, flexibility and numerical efficiency.
\end{remark}

\section{Main results}

Set $\nu_n(t):= \max\{k\geq 0 : t \geq t_{n,k}\}$, $t\in [0,1]$, and introduce an auxiliary pure jump process 
    \begin{equation}\label{defn:xl}
        \Xl(t) := \xi_{n,\nu_n(t)},\quad t \in [0,1].
    \end{equation}
Clearly, the trajectories of the process $\Xl$ belong to the Skorokhod space $\Df = \Df([0,1])$ which we will endow with the first Skorokhod metric
    \begin{equation*}
        d(x,y) = \inf_{\lambda \in \Lambda}\Big\{\varepsilon\geq 0: \sup_{t\in[0,1]} \lvert x(t) - y(\lambda(t)) \rvert \leq \varepsilon, \sup_{t\in[0,1]}\lvert\lambda(t)-t \rvert \leq \varepsilon\Big\},\quad x,y\in\Df,
    \end{equation*}    
where $\Lambda$ denotes the class of strictly increasing continuous mappings of $[0,1]$ onto itself (see, e.g.,~Ch.\,3 in \cite{bill}). We will use $\dto$ to denote convergence in distribution of random elements of $(\Df, d)$.

\begin{theorem}\label{thm:gauss_conv}
$X_n \dto X$ as $n\to \infty$.
\end{theorem}

Due to the small amplitude of the interpolating Brownian bridges' oscillations, it is unsurprising that the sequence of processes $\{\Xbb\}$ also converges weakly to $X$.

\begin{corollary}\label{corr:gauss_bb_conv}
$\Xbb \dto X$ as $n\to\infty$.
\end{corollary}

The following result is a theoretical justification of the Markov chain approximation method proposed in this paper.

\begin{corollary}\label{thm:sde_bcp_conv}
Assume that condition {\em{(C3)}} holds. Let $(g^-,g^+)$, $(g_n^-,g_n^+)$, $n\geq 1$, be elements of $\mathcal{G}$ such that $\norms{g_n^{\pm} - g^{\pm}}_{\infty} \to 0$ as $n\to \infty$, $G_n := S(g_n^-,g_n^+)$. Then, for any Borel set $B$ with $\partial B$ of zero Lebesgue measure,
    \begin{equation*}
         \lim_{n\to\infty}\p(\Xbb \in G_n, \Xbb(1) \in B)= \p(X \in G, X(1) \in B).
    \end{equation*}
\end{corollary}

The next result can be used, for instance, for computing risk-neutral prices of barrier options (see, e.g., \cite{borov}). It follows immediately from the previous corollary.  
\begin{corollary}
For any piecewise continuous $f : [g^-(1),g^+(1)] \to \mathbb{R}$, 
    \begin{equation*}
        \lim_{n \to \infty}\e f(\Xbb(1))\bm{1}\{\Xbb \in G_n\} = \e f(X(1))\bm{1}\{X \in G\}.
    \end{equation*}
\end{corollary}

\begin{remark}
To accelerate the numerical evaluation of $\p(\Xbb \in G_n)$, we can ignore the normalising factors $\C(x)$ as they are very close to 1. Indeed, let $\widehat{\mathbf{T}}_{n,k}$ be $\lvert E_{n,k-1}^G\rvert \times \lvert E_{n,k}^G\rvert $ matrices with entries
    \begin{equation}\label{defn:q}
        q_{n,k}(x,y) :=(1 - \pbb_{n,k}(f_n^-,f_n^+\,|\, x,y))\varphi(y\,|\, x + \driftGauss_{n,k}(x), \varianceGauss_{n,k}(x))h_{n,k}
    \end{equation}
for $(x,y) \in E_{n,k-1}^G\times E_{n,k}^G$. These matrices differ from the $\mathbf{T}_{n,k}$'s from \eqref{eqn:bcp_gauss} in that they don't involve the factors $\C(x)$. For $M$ and $c_0$ defined in Lemma \ref{lemma:C_sde}, we show below that
    \begin{equation}\label{drop_C}
        \big\lvert \p(\Xbb \in G_n) - \widehat{\mathbf{T}}_{n,1}\widehat{\mathbf{T}}_{n,2}\cdots \widehat{\mathbf{T}}_{n,n}\bm{1}^{\top} \big\rvert \leq c_0n\rho_n^{n-1}\exp\{-Mn^{2\delta}\} ,
    \end{equation}
where $\rho_n := 1 \vee \max_{1\leq j\leq n}\sup_{\lvert x\rvert \leq r} C_{n,j}(x)$. Note that some care must be exercised when choosing $\gamma$ and $\delta$. Using small values of $\gamma$ and $\delta$ will result in the error from replacing the normalising factors with 1 becoming non-negligible. It will be seen from the derivation of (3.3) that when $\delta =0$ the leading term of the error of the left hand side is equal to $2ne^{-2\pi^2\gamma^2}$. Figure 1 illustrates this observation numerically for the boundaries
    \begin{equation}\label{gpm}
        g^{\pm}(t) = \pm \mbox{$\frac{t}{3}$}\cosh^{-1}(2e^{9/(2t)}).
    \end{equation}
\begin{figure}[ht]
\centering
   \includegraphics{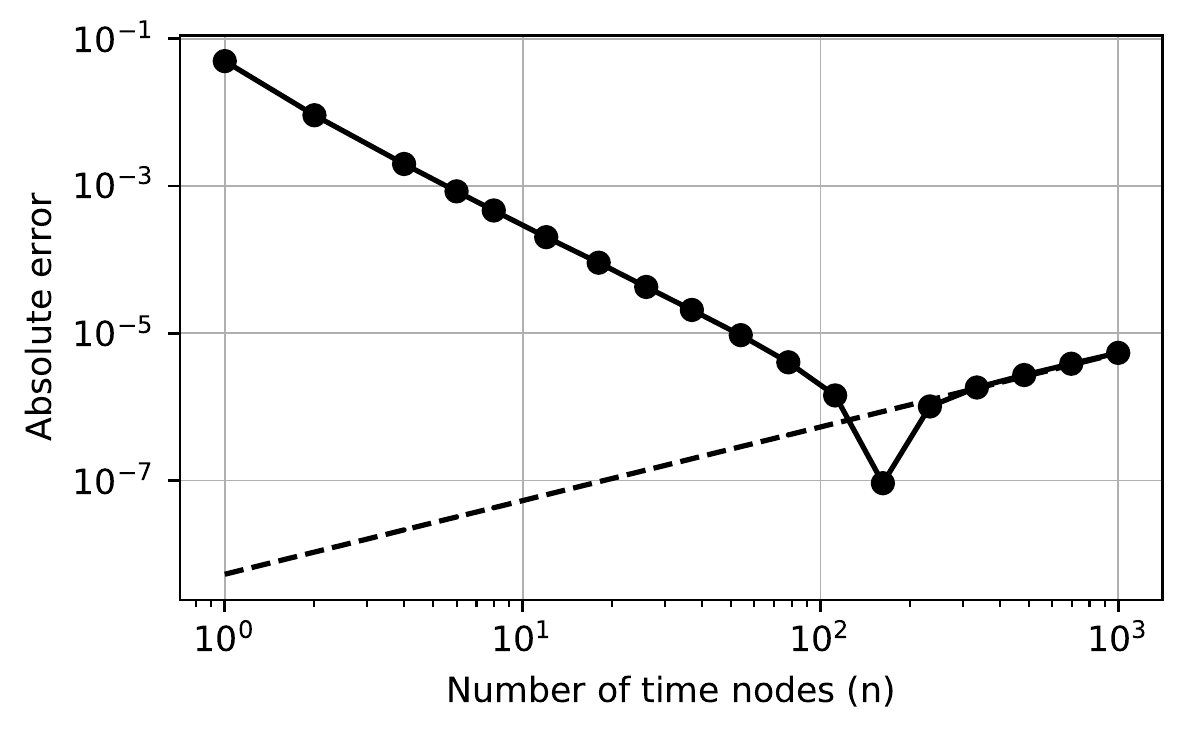}%
   \caption{On this log--log plot, the bullets $\bullet$ show the absolute errors between Markov chain approximations with $\gamma =1$, $\delta =0$ and normalising factors $\C(x)$ replaced with 1, and the true boundary crossing probability of the boundaries from \eqref{gpm}. The dashed line corresponds to the straight line $2ne^{-2\pi^2\gamma^2}$.}
   \label{figure_divergence}
\end{figure}
\end{remark}

\section{Proofs}

We will need the lemmata below to prove Theorem \ref{thm:gauss_conv}. Recall the quantity $K_r$ that appeared in condition (C3) and also notations $\maxTimeSpan$, $\driftZeta$ and $\stddevZeta$ from \eqref{eqn:d2}, \eqref{defn:b} and \eqref{defn:sigma}.  
\begin{lemma}\label{lemma:driftdiffusionbound}
For $r>\lvert x_0\rvert $,
    \begin{equation*}
         \adjustlimits\max_{1\leq k\leq n} \sup_{\lvert x\rvert\leq r} \big(\lvert \driftZeta_{n,k}(x) \rvert \vee   \varianceZeta_{n,k}(x) \big)  \leq K_r' ,
    \end{equation*}
where $K_r' := (\mbox{$K_r + \frac{\maxTimeSpan}{2}K_r(\frac{3}{2} + K_r )$}) \vee (1 + \mbox{$\frac{\maxTimeSpan}{2}K_r$})^2$.
\end{lemma}

The result is directly obtained by substituting the upper bound on the partial derivatives of $\driftX$ from condition (C3) and inequality \eqref{eqn:d2} into the definitions of $\driftZeta$ and $\varianceZeta$ in \eqref{defn:b} and \eqref{defn:sigma} respectively.

\begin{lemma}\label{lemma:C_sde}
For any $r>\lvert x_0 \rvert $ and $n > (1 + 2 K_r)\maxTimeSpan$, for the normalising factors \eqref{defn:scalingFactor} one has
    \begin{equation*}
        \adjustlimits\max_{1\leq k \leq n}\sup_{x \in E_{n,k}(r)}\lvert \C(x) - 1 \rvert \leq c_0e^{-Mn^{2\delta}},
    \end{equation*}
where $\delta$ is the quantity appearing in \eqref{eqn:h}, $M=M(\gamma,\delta,\eta_2) := \gamma^2\pi^2/(4\maxTimeSpan^{2\delta})$ and $c_0 := 2/(1 - e^{-M})$.
\end{lemma}

\begin{proof}[Proof of Lemma \ref{lemma:C_sde}]
For brevity, let $\mu := \driftGauss_{n,k}(x)$ and $\sigma^2  := \varianceGauss_{n,k}(x)$. Set 
    \begin{equation}\label{eqn:fourier_transform}
    \widehat{\varphi}_{\sigma}(s) := \int_{-\infty}^{+\infty}\varphi(x\,|\,0,\sigma^2)e^{-2\pi i s x}\,dx = e^{-2\pi^2 s^2\sigma^2} ,\quad s\in \mathbb{R}.
    \end{equation}
Since both $\varphi(\,\cdot\,\,|\, 0,\sigma^2)$ and $\widehat{\varphi}_{\sigma}(\,\cdot\,)$ decay at infinity faster than any power function, we can apply the Poisson summation formula (see, e.g.,~p.\,252 of \cite{Stein1971}) to obtain
        \begin{align*}
                \C(x) &= \sum_{j\in\mathbb{Z}}\varphi(jh_{n,k}+g_k^+\,|\,x +\mu,\sigma^2)h_{n,k} \\
                    &=1 + \sum_{\ell\in\mathbb{Z}\setminus \{0\}} \widehat{\varphi}_{\sigma}(\ell/h_{n,k}) e^{-2\pi i  (x + \mu -g_k^+)\ell/h_{n,k}}\\
                    &=  1 + 2 \sum_{\ell=1}^{\infty} \widehat{\varphi}_{\sigma}(\ell/h_{n,k}) \cos(2\pi( x+\mu-g_k^+) \ell/h_{n,k}),\quad x \in E_{n,k-1},
        \end{align*}
since $\widehat{\varphi}_{\sigma }(0) = 1$. It follows that
            \begin{equation*}
                \abs{\C(x) - 1} \leq  2\sum_{\ell=1}^{\infty}\lvert \widehat{\varphi}_{\sigma}(\ell/h_{n,k})\rvert.
            \end{equation*}
        Using \eqref{eqn:fourier_transform} and the elementary inequality $e^{-a\ell^2} \leq e^{-a\ell}$, $\ell\geq 1$, $a>0$, we obtain
            \begin{equation}\label{ineq:c1}
                    \abs{\C(x) - 1}\leq 2\sum_{\ell=1}^{\infty} e^{-2\pi^2 \ell \sigma^2 /h_{n,k}^2}=  \frac{2e^{-2\pi^2  \sigma^2 /h_{n,k}^2}}{1-e^{-2\pi^2\sigma^2/h_{n,k}^2}}.
            \end{equation}
        By condition (C3), for $n > 2\maxTimeSpan K_r$ we have
            \begin{align}\label{ineq:lowerBoundZetaVar}
                \inf_{\lvert x\rvert \leq r}\varianceZeta_{n,k}(x) &= \inf_{\lvert x\rvert \leq r}(1+ \half\Delta t_{n,k} \partial_x \driftX(t_{n,k-1},x))^2 \nonumber\\
                &\geq 1- \Delta t_{n,k}\sup_{\lvert x\rvert \leq r}\lvert \partial_x \driftX(t_{n,k-1},x) \rvert \geq 1 - \maxTimeSpan K_r/n \geq \half.
            \end{align}
        Substituting $\sigma^2 =\stddevGauss_{n,k}^2(x) = \varianceZeta_{n,k}(x)\Delta t_{n,k} \geq \Delta t_{n,k} /2$ into relation \eqref{ineq:c1} and using the inequality 
        \begin{equation}\label{ineq:deltaDivhSq}
            \min_{1\leq k\leq n}\frac{\Delta t_{n,k}}{h_{n,k}^{2}} =  \frac{1}{\scale_{n,n}^2\Delta t_{n,n}} \wedge \min_{1\leq k\leq n-1}\frac{1}{\scale_{n,k}^2(\Delta t_{n,k})^{2\delta}}\geq \frac{\gamma^2}{4}\left( \frac{n}{\maxTimeSpan} \wedge \bigg(\frac{n}{\maxTimeSpan}\bigg)^{2\delta}\right)  = \frac{\gamma^2n^{2\delta}}{4\maxTimeSpan^{2\delta}},
        \end{equation}
that holds since $\max_{1\leq k\leq n}\lvert \scale_{n,k}\rvert \leq 2\gamma^{-1}$, $\delta < 1/2$ we obtain
            \begin{equation*}
                \adjustlimits\max_{1\leq k \leq n }\sup_{\lvert x\rvert\leq r}\abs{\C(x) - 1}\leq \frac{2e^{-Mn^{2\delta}}}{1 - e^{-Mn^{2\delta}}}\leq c_0 e^{-Mn^{2\delta}}.
            \end{equation*}
\end{proof}

\begin{lemma}\label{lemma:moment_gauss_sde}
Let $Z$ be a standard normal random variable independent of $\{\xi_{n,k}\}_{k=1}^n$. Set
\begin{equation}\label{defn:dtaylor}
        \dtaylor(x):=  \driftGauss_{n,k}(x) + \stddevGauss_{n,k}(x)Z,\quad k=1,\ldots,n,\quad  x \in \mathbb{R},
\end{equation}
and let $\Delta\xi_{n,k}:=\xi_{n,k}-\xi_{n,k-1}$. Denote by $\e_{k,x}$ the conditional expectation given $\xi_{n,k-1}=x$, put $E_{n,k}(r) := E_{n,k-1}\cap [-r,r]$, $r >0$, and set, for $r>\lvert x_0\rvert $,  
    \begin{equation}\label{defn:e}
        e_{n,k}(m,r) := \sup_{x \in E_{n,k}(r)}\lvert \e_{k,x}(\Delta \xi_{n,k})^m - \e\dtaylor^m(x)\rvert ,\quad m= 1,2,\ldots 
    \end{equation}
Then
    \begin{equation*}
        \max_{1\leq k\leq n} e_{n,k}(m,r) \leq c_me^{-Mn^{2\delta}} \text{ for }n > (1+2 K_r)\maxTimeSpan \vee M^{-1/(2\delta)},
    \end{equation*}
where $c_m\in (0,\infty)$ is a constant whose explicit value is given below in \eqref{eqn:c_mr}.
\end{lemma}

\begin{proof}[Proof of Lemma \ref{lemma:moment_gauss_sde}]
        For brevity, we will often supress dependence on $m$, $n$, $k$ and $x$. This should cause no confusion. Set $C = \C(x)$ and 
        \begin{equation*}
            \momentdxi := \e_{k,x}(\dxi_{n,k})^m,\quad \momentdtaylor:= \e \dtaylor^m(x),\quad m= 1,2,\ldots
        \end{equation*}
        Using the triangle inequality, 
        \begin{equation}\label{eqn:1}
                  \lvert \momentdxi -  \momentdtaylor \rvert = \lvert \mbox{$\frac{1-C}{C}$}\momentdtaylor  +\mbox{$\frac{1}{C}$}(C \momentdxi - \momentdtaylor)  \rvert 
                    \leq \mbox{$\lvert \frac{ C - 1  }{C} \rvert\lvert$} \momentdtaylor\rvert  + \mbox{$\frac{1}{C}$} \lvert C\momentdxi - \momentdtaylor \rvert  .
            \end{equation}
        The term $C\momentdxi = C \e_{k,x} (\dxi_{n,k})^m$ can be viewed as a trapezoidal approximation of $\momentdtaylor=\e \dtaylor^m(x)$, so after rewriting $\momentdtaylor$ as an integral, we can express $C\momentdxi - \momentdtaylor$ as the quadrature error
            \begin{align*}
                 \varepsilon_{k}(x) :=  C\momentdxi - \momentdtaylor
                 =&\sum_{j \in \mathbb{Z}} (jh_{n,k} + g_k^+ -  x)^m\varphi(jh_{n,k} + g_k^+- x\,|\, \mu, \sigma^2) h_{n,k}  \nonumber\\
                    &\qquad  -\int_{-\infty}^{+\infty}(u + g_k^+-x)^m\varphi(u+g_k^+-x\,|\,\mu,\sigma^2)\,du,
            \end{align*}
        where $\mu = \driftGauss_{n,k}(x)$ and $\sigma^2 = \stddevGauss_{n,k}^2(x)$. We further note that, due to \eqref{defn:sigma}, the condition $n \geq 2\maxTimeSpan K_r$ ensures that $\sigma^2 >0$. Set 
    \begin{align}\label{eqn:fourier_transform_2}
    \FourierMoment(s) &:= \int_{-\infty}^{+\infty}(v + \mu)^m\varphi(v\,|\,0,\sigma^2)e^{-2\pi i s v}\,dv \nonumber\\
    &= \sum_{l=0}^m \binom{m}{l}\mu^{m-l}\int_{-\infty}^{+\infty}v^l\varphi(v\,|\, 0 , \sigma^2)e^{-2\pi i s v}\,dv\nonumber \\
    &= \sum_{l=0}^m \binom{m}{l}\mu^{m-l} (-i\sigma)^lH_l(2\pi s  \sigma)e^{-2\pi^2 s^2 \sigma^2},\quad s\in \mathbb{R},
    \end{align}
where $H_l(x) :=(-1)^le^{x^2/2}\frac{d^l}{dx^l}e^{-x^2/2}$, $x\in \mathbb{R}$, $l \geq 1$, is the $l$-th Chebyshev--Hermite polynomial. Since both $(\,\cdot\,-x)^m \varphi(\,\cdot\, - x\,|\,\mu,\sigma^2)$ and $\FourierMoment(\,\cdot\,)$ decay at infinity faster than any power function, using the Poisson summation formula (see, e.g.,~p.\,252 of~\cite{Stein1971}) and the change of variables $v = u h_{n,k} + g_k^+ -x $ in \eqref{eqn:fourier_transform_2}, we obtain
            \begin{align*}
                 \varepsilon_{k}(x) &= \sum_{\ell\in\mathbb{Z}\backslash \{0\}} \FourierMoment(\ell/h_{n,k})e^{-2\pi i(-g_k^ + + x + \mu)\ell/h_{n,k}}\\
                 &= 2\sum_{\ell= 1}^{\infty}
                 \Re\bigl[\FourierMoment(\ell/h_{n,k})e^{-2\pi i(-g_k^+ + x +\mu)\ell/h_{n,k}}\bigr].
            \end{align*}
        Since $\lvert \Re z \rvert\leq \lvert z\rvert$, $z \in \mathbb{C}$, and $\lvert e^{is}\rvert\leq 1$, $s \in \mathbb{R}$, we obtain
            \begin{equation*}
                \lvert \varepsilon_{k}(x)\rvert \leq
                    2\sum_{\ell= 1}^{\infty}\lvert\FourierMoment(\ell/h_{n,k})\rvert.
            \end{equation*}
        Note that $\lvert H_l(u) \rvert \leq C_l'(\lvert u\rvert^{\ell} + 1)$, $u \in \mathbb{R}$, $l =1,2,\ldots,$ where $C_l'$ is a constant that depends on $l$ only and we can assume without loss of generality that $\{C_l'\}_{l\geq 1}$ is a non-decreasing sequence. Therefore one has $\lvert (-i \sigma )^l H_l(2 \pi s \sigma )\rvert \leq C_l( s^l \sigma^{2l} +\sigma^l)$, $s \geq 1$, where $C_l:= (2\pi)^lC_l'$. Using $\mu^{m-l}(\ell\sigma^2/h_{n,k})^l= \beta_{n,k}^{m-l}(x)\alpha_{n,k}^l(x)\ell^m \Delta t_{n,k}^m/(h_{n,k}^l\ell^{m-l})$, and the fact that $h_{n,k}^{-l}\ell^{-(m-l)}\leq h_{n,k}^{-m}$ for $l\leq m$ as $h_{n,k}<1$, we obtain from \eqref{eqn:fourier_transform_2}
            \begin{align*}
              \lvert \varepsilon_{k}(x) \rvert  &\leq 2\sum_{\ell =1}^{\infty}\sum_{l=0}^m \binom{m}{l} C_l \lvert \mu\rvert ^{m-l} \bigg( \bigg(\frac{\ell \sigma^2}{h_{n,k}}\bigg)^l + \sigma^l \bigg)e^{-2\pi^2 \sigma^2 \ell^2/h_{n,k}^2} \\
                &\leq  2 \sum_{\ell=1}^{\infty}\sum_{l=0}^m \binom{m}{l}C_l\bigg( \bigg(\frac{\Delta t_{n,k}}{h_{n,k}}\bigg)^m \lvert \driftZeta_{n,k}^{m-l}(x)\rvert \varianceZeta_{n,k}^{l}(x) \ell^m \\
                &\qquad\qquad\qquad\qquad \qquad  +\lvert \driftZeta_{n,k}^{m-l}(x)\rvert \varianceZeta_{n,k}^{l/2}(x)(\Delta t_{n,k})^{m-l/2}\bigg)e^{-2\pi^2 \sigma^2 \ell^2 /h_{n,k}^2}.
             \end{align*}
        Since $C_l$ is non-decreasing in $l$, one has 
            \begin{equation*}
            \sup_{x \in E_{n,k}(r)}\sum_{l=0}^m\binom{m}{l}C_l\bigg(\frac{\Delta t_{n,k}}{h_{n,k}}\bigg)^m \lvert \beta_{n,k}^{m-l}(x)\rvert \alpha_{n,k}^l(x)\leq  2^mC_m\bigg(\frac{2K_r'\maxTimeSpan}{\minTimeSpan}\bigg)^m =:L_{m,r},
            \end{equation*}
        where we used Lemma \ref{lemma:driftdiffusionbound} and the inequality $\max_{1\leq k\leq n}\Delta t_{n,k}/h_{n,k} \leq  2\maxTimeSpan/\minTimeSpan$, which follows from $\min_{1\leq k\leq n}h_{n,k}\geq \min\{ (\minTimeSpan/n)^{1/2 + \delta},\minTimeSpan/n\}/2=  \minTimeSpan/(2n)$. Again using Lemma \ref{lemma:driftdiffusionbound} and the trivial bound $(\Delta t_{n,k})^{m-l/2}\leq 1$, $0\leq l\leq m$, we have
            \begin{equation*} \sup_{x \in E_{n,k}(r)}\sum_{l=0}^{m}\binom{m}{l}C_l \lvert \beta_{n,k}^{m-l}(x)\rvert \alpha_{n,k}^{l/2}(x)(\Delta t_{n,k})^{m-l/2}\leq 2^m C_m(K_r')^m \leq L_{m,r}.
            \end{equation*}
        Hence 
            \begin{align*}
               \sup_{x \in E_{n,k}(r)}\lvert\varepsilon_k(x) \rvert &\leq 2L_{m,r}\sum_{\ell =1}^{\infty} (\ell^m + 1)e^{-2\pi^2\sigma^2 \ell^2/h_{n,k}^2} \leq 4L_{m,r}\sum_{\ell =1}^{\infty} \ell^m e^{-\ell M n^{2\delta}}, 
            \end{align*}
        where we used \eqref{ineq:lowerBoundZetaVar},  \eqref{ineq:deltaDivhSq} and the bound $e^{-a\ell^2} \leq e^{-a\ell}$, $\ell \geq 1$, $a>0$, in the second inequality. Note that   
            \begin{equation*}
         \sum_{\ell=1}^{\infty}\ell^mz^{\ell}\leq a_mz,\quad z \in [0,e^{-1}],
            \end{equation*}
        where $a_m:= \sum_{\ell =1}^{\infty}\ell^m e^{-\ell +1} <\infty$. As $Mn^{2\delta}>1$, we obtain from here that
            \begin{equation}\label{ineq:quadrature_moments}
                \adjustlimits\max_{1\leq  k \leq n }\sup_{x \in E_{n,k}(r)}\lvert \varepsilon_{k}(x)\rvert\leq 4a_mL_{m,r} e^{-Mn^{2\delta}}.
            \end{equation}
From Lemma \ref{lemma:C_sde},
            \begin{equation}\label{ineq:final}
                \sup_{x \in E_{n,k}(r)}\left|\frac{ \C(x) -1}{\C(x)} \right| \leq  \frac{c_0e^{-Mn^{2\delta}}}{1-c_0}.
            \end{equation}
        Using inequalities \eqref{ineq:quadrature_moments} and \eqref{ineq:final} in \eqref{eqn:1}, we get
            \begin{equation*}
               \adjustlimits\max_{1\leq k\leq n}\sup_{x \in E_{n,k}(r)}\lvert e_{n,k}(m,r) \rvert  \leq c_me^{-Mn^{2\delta}},
            \end{equation*}
        where
            \begin{equation}\label{eqn:c_mr}
                c_m := \frac{c_0}{1-c_0}\adjustlimits\max_{1\leq k\leq n}\sup_{x \in E_{n,k}(r)}\lvert\e\dtaylor^m(x)\rvert + \frac{4a_mL_{m,r}}{1-c_0}.
            \end{equation} 
        The boundedness of $\max_{1\leq k\leq n}\sup_{x \in E_{n,k}(r)}\lvert \e \dtaylor^m(x)\rvert$ can be proved by applying the inequality $\lvert x +y \rvert^m \leq 2^{m-1}(\lvert x \rvert^m + \lvert y\rvert^m)$, $x,y \in\mathbb{R}$, $m\geq 1$, to obtain
            \begin{align*}
                \sup_{x \in E_{n,k}(r)}\lvert\e \dtaylor^m(x)\rvert &\leq 2^{m-1} \sup_{\lvert x\rvert  \leq r}(\lvert \driftZeta_{n,k}(x)\Delta t_{n,k}\rvert^m + \lvert \varianceZeta_{n,k}(x)\Delta t_{n,k}\rvert^{m/2} \e \lvert Z\rvert^m ) \\
                    &\leq 2^{m-1}( (\maxTimeSpan K_r' )^m + (2\maxTimeSpan K_{r}')^{m/2} \pi^{-1/2}\mbox{$\Gamma(\frac{m+1}{2}$}) ) < \infty,
            \end{align*}
        where $\Gamma$ is the gamma function.
    \end{proof}
To prove the convergence stated in Theorem 1 we will use the martingale characterisation method, verifying the sufficient conditions for convergence from Theorem~4.1 in Ch.\,7 of~\cite{ethier09} (to be referred to as the EK theorem in what follows). For $x \in E_{n,k-1}$, let
\begin{equation*}
    \begin{aligned}
    b_{n,k}(x) &:= \mbox{$\frac{1}{\Delta t_{n,k}}$}\e[\Delta\xi_{n,k}\,|\, \xi_{n,k-1} =x],\\
    \varianceXi_{n,k}(x) &:= \mbox{$\frac{1}{\Delta t_{n,k}}$}\var[\Delta\xi_{n,k}\,|\, \xi_{n,k-1} =x].
    \end{aligned}
\end{equation*}
Using the standard semimartingale decomposition of $X_n$ (see \eqref{defn:xl}), we set 
    \begin{equation*}
        \begin{aligned}
        B_n(t) &:=   \mbox{$\sum_{k=1}^{\nu_n(t)}$}b_{n,k}(\xi_{n,k-1})\Delta t_{n,k},\\
        A_n(t) &:= \mbox{$\sum_{k=1}^{\nu_n(t)}$}\varianceXi_{n,k}(\xi_{n,k-1})\Delta t_{n,k},
        \end{aligned}
    \end{equation*}
and let $M_n:= X_n - B_n$. With respect to the natural filtration
\begin{equation*}
    \mathbf{F}^n:= \{ \sigma(X_n(s),B_n(s),A_n(s) : s\leq t) : t\geq 0\} = \{\sigma(\xi_{n,k} : k \leq \nu_n(t),t\geq 0 )\},    
\end{equation*}
our $B_n$, $A_n$ and $M_n$ are the predictable drift, angle bracket and martingale component, respectively, of the process $\Xl$. 

The EK theorem is stated for time-homogeneous processes. To use it in our case, we consider the vector-valued processes $\bm{X}_n(t) := (t,X_n(t))$ and, for a fixed $r>0$, let $\tau_n^r$ be localising $\mathbf{F}^n$-stopping times:
    \begin{equation*}
        \tau_n^r := \inf\{t : \norms{\bm{X}_n(t)} \vee \norms{\bm{X}_n(t-) }\geq r \},
    \end{equation*}
$\norms{\bm{u}} = \lvert u_1\rvert \vee \lvert u_2 \rvert$ being the maximum norm of $\bm{u} =(u_1,u_2)$. 

\begin{lemma}\label{characteristics_converge}
For each fixed $r>\lvert x_0\rvert $, as $n\to\infty$,
    \begin{equation*}
        \sup_{t \leq 1 \wedge \tau_n^r }\Big\lvert  B_n(t) - \int_0^t \driftX(s,X_n(s))\,ds\Big\rvert \xrightarrow{\as} 0,\quad
        \sup_{t \leq 1  \wedge \tau_n^r}\lvert  A_n(t) - t\rvert \xrightarrow{\as} 0.
    \end{equation*}

\end{lemma}

\begin{proof}[Proof of Lemma \ref{characteristics_converge}]         
Since on each of the time intervals $[t_{n,k-1},t_{n,k})$, $k=1,\ldots,n,$ the process $X_n$ is equal to $\xi_{n,k-1}$, we have the following decomposition:
         \begin{align}\label{eqn:b_error}
                &B_n(t) - \mbox{$\int_0^t$} \driftX(s,X_n(s))\,ds \nonumber\\
                &= \mbox{$\sum_{k=1}^{\nu_n(t)}$}\bigl[ \driftspace_{n,k}(\xi_{n,k-1})\Delta t_{n,k} - \mbox{$\int_{t_{n,k-1}}^{t_{n,k}}$} \driftX(s,\Xl(s))\,ds\bigr] - \mbox{$\int_{t_{n,\nu_n(t)}}^{t}$}\driftX(s,\Xl(s))\,ds\nonumber\\
            &=\mbox{$\sum_{k=1}^{\nu_n(t)}$}( \driftXi_{n,k}(\xi_{n,k-1}) - \driftZeta_{n,k}(\xi_{n,k-1}))\Delta t_{n,k} \nonumber\\
                &\quad +\mbox{$\sum_{k=1}^{\nu_n(t)}$}( \driftZeta_{n,k}(\xi_{n,k-1}) - \driftX(t_{n,k-1},\xi_{n,k-1}) )\Delta t_{n,k} \nonumber\\
                    &\quad + \mbox{$\sum_{k=1}^{\nu_n(t)}$}\mbox{$\int_{t_{n,k-1}}^{t_{n,k}}$} [\driftX(t_{n,k-1},X_n(s))-\driftX(s,X_n(s)) ]\,ds\nonumber \\
                    &\quad - \mbox{$\int_{t_{n,\nu_n(t)}}^{t}$}\driftX(s,\xi_{n,\nu_n(t)})\,ds.
            \end{align}
Due to the stopping-time localisation, the first term on the right-hand side of \eqref{eqn:b_error} has the following upper bound (see \eqref{defn:mu_and_sigma}, \eqref{defn:dtaylor} and \eqref{defn:e}):
        \begin{align}\label{ineq:b1}
            &\sup_{t \leq 1 \wedge \tau_n^r}\Big\lvert  \mbox{$\sum_{k=1}^{\nu_n(t)}$}(\driftspace_{n,k}(\xi_{n,k-1}) - \driftZeta_{n,k}(\xi_{n,k-1}))\Delta t_{n,k}\Big\rvert \nonumber\\
            &= \sup_{t \leq 1\wedge \tau_n^r}\Big\lvert\mbox{$\sum_{k=1}^{\nu_n(t)}$} (\e [\Delta \xi_{n,k}\,|\, \xi_{n,k-1}] - \e [\dtaylor(\xi_{n,k-1})\,|\, \xi_{n,k-1}])\Big\rvert\nonumber\\
            &\leq \mbox{$\sum_{k=1}^n$}\sup_{x \in E_{n,k}(r)}\lvert \e_{k,x}\Delta \xi_{n,k}-\e  \dtaylor(x) \rvert \nonumber\\
            &\leq \mbox{$\sum_{k=1}^n$} e_{n,k}(1,r).
        \end{align}
To bound the second term on the right-hand side of \eqref{eqn:b_error}, we use the definition of $\beta_{n,k}$ in \eqref{defn:b} and
condition (C3) to get following inequality:
            \begin{align}\label{ineq:b2}
                \sup_{x \in E_{n,k}(r)}\lvert \driftZeta_{n,k}(x) - \driftX(t_{n,k-1},x) \rvert &\leq \half\Delta t_{n,k}\sup_{\lvert x\rvert \leq r }\lvert (\partial_t\driftX  + \driftX \partial_x \driftX+ \half \partial_{xx} \driftX)(t_{n,k-1},x)\rvert \nonumber\\
                    &\leq \half \Delta t_{n,k}K_r(K_r + \mbox{$\frac{3}{2}$}) .
            \end{align}
    The second last term in \eqref{eqn:b_error} can be bounded from above by using the bound for $\partial_t \driftX$ from (C3):
     \begin{equation}\label{ineq:b3}
                \sup_{t \leq 1\wedge \tau_n^r}\sum_{k=1}^{\nu_n(t)}\left\lvert \mbox{$\int_{t_{n,k-1}}^{t_{n,k}}$}[\driftX(t_{n,k-1},X_n(s))-\driftX(s,X_n(s)) ]\,ds\right\rvert
                    \leq K_r\max_{1\leq k \leq n}\Delta t_{n,k}.
            \end{equation}
        Again using (C3), the last term in \eqref{eqn:b_error} is bounded as follows: 
            \begin{equation}\label{ineq:b4}
                \sup_{t\leq 1 \wedge \tau_n^r}\mbox{$\left\lvert \int_{t_{n,\nu_n(t)}}^t \driftX(s,X_{n,\nu_n(t)})\,ds\right\rvert $}\leq K_r\max_{1\leq k \leq n}\Delta t_{n,k}.
            \end{equation}
          Using inequalities \eqref{ineq:b1}--\eqref{ineq:b4} in the decomposition \eqref{eqn:b_error}, we obtain
            \begin{align*}
                 \sup_{t \leq 1 \wedge \tau_n^r}\left\lvert B_n(t) - \mbox{$\int_0^t$} \driftX(s,X_n(s))\,ds \right\rvert \leq \sum_{k=1}^n e_{n,k}(1,r) + \mbox{$K_r(\half K_r +\frac{11}{4})$}\max_{1 \leq k\leq n}\Delta t_{n,k}.
            \end{align*}
        Since $\max_{1\leq k \leq n}\Delta t_{n,k}\leq \maxTimeSpan/n\to 0$ as $n\to \infty$, the first half of the lemma follows after applying Lemma \ref{lemma:moment_gauss_sde} with $m=1$.
        
Similarly,
    \begin{align}\label{eqn:quadvar}
                A_n(t) - t &=
                \mbox{$\sum_{k=1}^{\nu_n(t)}$}( \varianceXi_{n,k}(\xi_{n,k-1}) - \varianceZeta_{n,k}(\xi_{n,k-1}))\Delta t_{n,k}\nonumber \\
                &\quad + \mbox{$\sum_{k=1}^{\nu_n(t)}$}(\varianceZeta_{n,k}(\xi_{n,k-1}) - 1) \Delta t_{n,k} - \mbox{$\int_{t_{n,\nu_n(t)}}^{t}$} \,ds.
    \end{align}    
To bound the first term on the right-hand side of \eqref{eqn:quadvar}, we use $\v [Y \,|\, X] = \e[Y^2\,|\, X] - (\e[Y\,|\, X])^2$ for a square-integrable random variable $Y$ to obtain
    \begin{align}\label{ineq:a1}
    &\sup_{t \leq 1 \wedge \tau_n^r}\Big\lvert \mbox{$\sum_{k=1}^{\nu_n(t)}$}( \varianceXi_{n,k}(\xi_{n,k-1}) - \varianceZeta_{n,k}(\xi_{n,k-1}))\Delta t_{n,k}\Big\rvert\nonumber\\
        &=\sup_{t \leq 1 \wedge \tau_n^r}\Big\lvert \mbox{$\sum_{k=1}^{\nu_n(t)}$}( \v [\Delta \xi_{n,k}\,|\, \xi_{n,k-1}] - \v[ \dtaylor( \xi_{n,k-1})\,|\xi_{n,k-1}]) \Big\rvert \nonumber\\
        &\leq \sup_{t \leq 1 \wedge \tau_n^r} \sum_{k=1}^{\nu_n(t)}\bigl\lvert  \e [(\Delta \xi_{n,k})^2\,|\, \xi_{n,k-1}] - \e [\dtaylor^2(\xi_{n,k-1})\,|\,\xi_{n,k-1}]\bigr\rvert   \nonumber\\ 
        &\qquad +\sup_{t \leq 1 \wedge \tau_n^r}\sum_{k=1}^{\nu_n(t)}\bigl\lvert (\e [\Delta \xi_{n,k}\,|\, \xi_{n,k-1}])^2 - (\e [\dtaylor( \xi_{n,k-1})\,|\,\xi_{n,k-1}])^2\bigr\rvert \nonumber\\
        &\leq \sum_{k=1}^{n} e_{n,k}(2,r) + 2\sum_{k=1}^ne_{n,k}(1,r)\Bigl(e_{n,k}(1,r) + \sup_{x \in E_{n,k}(r) }\lvert \e \dtaylor(x)\rvert\Bigr) ,
        \end{align}
 where we used the elementary bound
    \begin{equation}\label{ineq:squares}
        \lvert x^2 -y^2 \rvert \leq \lvert x - y\rvert (\lvert x-y\rvert + 2\lvert y\rvert)
    \end{equation}
 in the final inequality. Furthermore,
    \begin{equation*}
        \adjustlimits\max_{1\leq k \leq n}\sup_{x\in E_{n,k}(r)}\lvert \e \dtaylor(x)\rvert \leq \adjustlimits\max_{1\leq k \leq n}\sup_{\lvert x\rvert\leq r}\lvert \driftZeta_{n,k}(x)\rvert\Delta t_{n,k}\leq \maxTimeSpan K_r' .
    \end{equation*}
Lemma \ref{lemma:moment_gauss_sde} with $m=1$ and $m=2$ implies that the expression in the last line of \eqref{ineq:a1} vanishes as $n\to\infty$. The second term on the right-hand side of \eqref{eqn:quadvar} is bounded from above using \eqref{defn:sigma}, Lemma \ref{lemma:driftdiffusionbound} and condition (C3):
        \begin{align}\label{ineq:a2}
          \sup_{t\leq 1 \wedge \tau_n^r} \sum_{k=1}^{\nu_n(t)}\lvert \varianceZeta_{n,k}(\xi_{n,k-1}) - 1\rvert\Delta t_{n,k} &\leq  \adjustlimits\max_{1\leq k\leq n}\sup_{x\in E_{n,k}(r)}\lvert \stddevZeta_{n,k}(x) - 1\rvert (   \stddevZeta_{n,k}(x) +1)  \nonumber\\
                &\leq \mbox{$\frac{1}{2}$}\max_{1\leq k \leq n} \Delta t_{n,k}K_r  ( \sqrt{K_r'} +1).
        \end{align}
        For the last term in \eqref{eqn:quadvar} one has 
            \begin{align}\label{ineq:a3}
                \sup_{t\leq 1 \wedge \tau_n^r}\lvert \mbox{$\int_{t_{n,\nu_n}(t)}^t$} \,ds\rvert
                &\leq \max_{1\leq k \leq n}\Delta t_{n,k} .
            \end{align}
    Applying inequalities \eqref{ineq:a1}--\eqref{ineq:a3} to \eqref{eqn:quadvar}, the result follows.
\end{proof}

\begin{lemma}\label{tightness_SDE}
For each fixed $r>\lvert x_0\rvert $,
\begin{equation}\label{lemma:b_jump_conv}
                        \lim_{n\to\infty}\e \sup_{t\leq 1 \wedge \tau_n^r} \lvert B_n(t) - B_n(t-) \rvert^2 =0,
                    \end{equation}
    \begin{equation}\label{lemma:a_jump_conv}
                        \lim_{n\to\infty}\e \sup_{t\leq 1 \wedge \tau_n^r} \lvert A_n(t) - A_n(t-) \rvert =0,
                    \end{equation}
and                
\begin{equation}\label{lemma:x_jump_conv}                    \lim_{n\to\infty}\e \sup_{t\leq 1 \wedge \tau_n^r} \lvert X_n(t) - X_n(t-) \rvert^2 =0.
                    \end{equation}                
\end{lemma}

\begin{proof}[Proof of Lemma \ref{tightness_SDE}]
Set $\overline{\xi}_n := \max_{1\leq k\leq n}\lvert \xi_{n,k}\rvert$ and
    \begin{equation*}
        \chi_n^r := \min\{k\leq n :\lvert \xi_{n,k}\rvert\geq r\}\bm{1}\{\overline{\xi}_n \geq r\} + n\bm{1}\{\overline{\xi}_n <r\}.
    \end{equation*}
The jumps of $B_n$ are given by the conditional means of the increments, so
    \begin{align*}
        \e \sup_{t\leq 1 \wedge \tau_n^r} \lvert B_n(t) - B_n(t-) \rvert^2  &= \e \max_{ k \leq \chi_n^r }(\e[\dxi_{n,k}\,|\,\xi_{n,k-1}])^2 \\
        &\leq  \e \adjustlimits\max_{1\leq k \leq n }\sup_{x \in E_{n,k}(r)}(\e_{k,x}\Delta \xi_{n,k})^2.
    \end{align*}
By the triangle inequality, for $x \in E_{n,k}(r)$,
\begin{align*}
        \lvert \e_{k,x} \Delta \xi_{n,k} \rvert
        &\leq \lvert \e_{k,x} \Delta \xi_{n,k} - \e \dtaylor(x)\rvert  + \lvert \e\dtaylor(x)\rvert\\
        &\leq e_{n,k}(1,r) + \lvert \driftZeta_{n,k}(x)\rvert\Delta t_{n,k}.
    \end{align*}
By inequality \eqref{eqn:d2} and Lemmata 1 and 3, we obtain \eqref{lemma:b_jump_conv}. The jumps of $A_n$ are given by the conditional variances of the increments, so
    \begin{align*}
        \e \sup_{t\leq 1 \wedge \tau_n^r} \lvert A_n(t) - A_n(t-) \rvert  
        &= \e \max_{k\leq \chi_n^r}\v[\dxi_{n,k}\,|\,\xi_{n,k-1}]\\
        &\leq \e  \adjustlimits\max_{1\leq k \leq n }\sup_{x\in E_{n,k}(r)}\v_{k,x}\Delta\xi_{n,k},
    \end{align*}
where $\v_{k,x}[\,\cdot\,]:=\v[\,\cdot\,|\, \xi_{n,k-1}=x]$. Using \eqref{ineq:squares}, we obtain, for $x \in E_{n,k}(r)$,
    \begin{align*}
        \lvert \v_{k,x} \Delta \xi_{n,k}\rvert
        &\leq \lvert \v_{k,x} \Delta \xi_{n,k} - \v \dtaylor(x)\rvert  +\v \dtaylor(x)\\
        &\leq \lvert \e_{k,x} (\Delta \xi_{n,k})^2 - \e \dtaylor^2(x)\rvert + \lvert (\e_{k,x} \Delta \xi_{n,k})^2 - (\e \dtaylor(x))^2\rvert + \varianceGauss_{n,k}(x)\\
        &\leq e_{n,k}(2,r) + e_{n,k}(1,r)( e_{n,k}(1,r) + 2\lvert \driftZeta_{n,k}(x) \rvert \Delta t_{n,k}) + \varianceZeta_{n,k}(x)\Delta t_{n,k}.
    \end{align*}
     By the local boundedness of $\driftZeta_{n,k}$ and $\varianceZeta_{n,k}$, from Lemma 1 we obtain \eqref{lemma:a_jump_conv}.
Further,            
    \begin{equation*}
        \e \sup_{t\leq 1 \wedge \tau_n^r} \lvert X_n(t) - X_n(t-) \rvert^2 
            = \e \max_{k\leq \chi_n^r}(\dxi_{n,k})^2.
    \end{equation*}        
Using Lyapunov's inequality, we obtain 
    \begin{align*}
        \e \max_{k\leq \chi_n^r}(\dxi_{n,k})^2 &\leq (\e \max_{k\leq \chi_n^r}( \dxi_{n,k})^{4})^{1/2}\\
        &\leq (\e \mbox{$\sum_{k \leq \chi_n^r}$} (\dxi_{n,k})^4 \, )^{1/2}\\ 
        &\leq \mbox{$(\sum_{k=1}^n\sup_{x\in E_{n,k}(r)}\e_{k,x}( \Delta \xi_{n,k} )^{4})^{1/2}$}.
    \end{align*}
By the triangle inequality, we have
    \begin{align*}
        \lvert \e_{k,x} (\Delta \xi_{n,k})^{4} \rvert
        &\leq \lvert \e_{k,x} (\Delta \xi_{n,k})^4 - \e \dtaylor^4(x)\rvert  + \e \dtaylor^4(x) \\
        &\leq e_{n,k}(4,r) + (\driftZeta_{n,k}(x)\Delta t_{n,k})^4 \\
        &\quad + 6(\driftZeta_{n,k}(x)\Delta t_{n,k})^2\varianceZeta_{n,k}(x)\Delta t_{n,k} + 3( \varianceZeta_{n,k}(x)\Delta t_{n,k})^2.
    \end{align*}
Using Lemmata 1 and 3 we obtain \eqref{lemma:x_jump_conv}.
\end{proof}

\begin{proof}[Proof of Theorem \ref{thm:gauss_conv}]
We verify the conditions of the EK Theorem. Denote by $L$ the generator of the bivariate process $\bm{X} := \{\bm{X}(t) = (t,X(t))\}_{t\in [0,1]}$:
    \begin{equation*}
      Lf = \partial_t f+\mu \partial_xf + \half \partial_{xx}f,\quad f\in C_c^{\infty}(\mathbb{R}^2),
    \end{equation*}
where $C_c^{\infty}(\mathbb{R}^2)$ is the space of infinitely many times differentiable functions with compact support. The distribution of $\bm{X}$ is the solution to the martingale problem for $L$, i.e., for $f \in C_c^{\infty}(\mathbb{R}^2)$,
    \begin{equation*}
        f(\bm{X}(t)) - f(\bm{X}(0)) - \int_0^t Lf(\bm{X}(s))\,ds , \quad t\in[0,1],
    \end{equation*}
is a martingale. Using condition (C3), by Proposition 3.5 in Ch.\,5 of \cite{ethier09} the martingale problem for $L$ is well-posed since the solution for the stochastic differential equation \eqref{defn:sde_X} exists and is unique. Therefore the first condition in the EK Theorem is met.

The martingale characteristics of $\bm{X}_n = (t,\Xl(t))$ are given by
    \begin{equation*}
        \bm{B}_n(t) := (t,B_n(t)),\quad \bm{A}_n(t) := \begin{pmatrix} 0 & 0\\ 0 & A_n(t)\end{pmatrix},\quad  \bm{M}_n(t) := (0,M_n(t)).
    \end{equation*}
A simple calculation shows that $\bm{M}_n$ and $\bm{M}_n^{\top}\bm{M}_n - \bm{A}_n$ are $\mathbf{F}^n$-martingales. It follows from Lemmata \ref{characteristics_converge} and \ref{tightness_SDE} that conditions (4.3)--(4.7) in the EK Theorem are satisfied, which means that all the conditions of that theorem are met.
\end{proof}

 \begin{proof}[Proof of Corollary \ref{corr:gauss_bb_conv}]
Denote the process whose trajectories are polygons with nodes $(t_{n,k}, \xi_{n,k})$ by $\Xh$. By the triangle inequality, 
    \begin{equation*}
       \norms{\Xbb-\Xl}_{\infty}  \leq  \norms{\Xbb-\Xh}_{\infty} + \norms{\Xh-\Xl}_{\infty},
    \end{equation*}
where $\Xbb$ was defined in \eqref{defn:xbb}. Using the distribution of the maximum of the standard Brownian bridge $B^{\circ}$ (see, e.g.,~p.\,63 of \cite{borodin}), for any $\varepsilon>0$ we obtain
    \begin{align*}
        \p(\norms{\Xbb-\Xh}_{\infty} \geq \varepsilon) &= \p\Bigl(\adjustlimits\max_{1\leq k \leq n} \sup_{s \in[t_{n,k-1},t_{n,k}]} \lvert \bb_{n,k}(s)\rvert \geq \varepsilon\Bigr) \\
            &\leq \sum_{k=1}^n\p\Bigl(\sup_{t\in[0,1]} \lvert \bb(t)\rvert  \geq \varepsilon/\sqrt{\Delta t_{n,k}} \Bigr) \\           
            &\leq 2 n \exp\Bigl\{-2\varepsilon^2/\max_{1\leq k \leq n}\Delta t_{n,k}\Bigr\}.
    \end{align*}
Hence $\norms{\Xbb-\Xh}_{\infty} \xrightarrow{p} 0$ since $\max_{1 \leq k \leq n}\Delta t_{n,k}\leq \maxTimeSpan/n$. Further, $\norms{\Xh-\Xl}_{\infty} = \sup_{t \in (0,1]}\lvert X_n(t) -X_n(t-)\rvert\xrightarrow{p} 0$ since $\Xl \dto X$ and $X$ is almost surely continuous. Therefore $d(\Xbb,\Xl) \leq \norms{\Xbb - \Xl}_{\infty} \xrightarrow{p} 0$. It follows from Theorem 4.1 in \cite{bill} that $\Xbb \dto X$ as $n\to\infty$.
\end{proof} 
        
\begin{proof}[Proof of Corollary \ref{thm:sde_bcp_conv}]
For sets $A \subset \Cf$ and $B \subseteq \mathbb{R}$, set
    \begin{equation*}
        A[B] := A \cap \{ x \in \Cf : x(1) \in B\}.
    \end{equation*}
Recall that, for $(g^-,g^+)\in \mathcal{G}$, one has $\inf_{t\in [0,1]}(g^+(t) - g^-(t) ) >0$. For any $\varepsilon \in (0, \varepsilon')$, $\varepsilon':=[ \inf_{t\in [0,1]}(g^+(t) - g^-(t) )/2] \wedge ( g^+(0) -x_0) \wedge   (x_0 -g^-(0)),$ and all sufficiently large $n$ (such that $\norms{g_n^{\pm} - g^{\pm}}_{\infty} < \varepsilon'$), the ``strips'' $G^{\pm \varepsilon} := S(g^-\mp \varepsilon, g^{+}\pm \varepsilon)$ are non-empty and
    \begin{equation*}
        \p(\Xbb \in G^{-\varepsilon}[B]) \leq \p(\Xbb \in G_n[B]) \leq \p(\Xbb \in G^{+\varepsilon}[B]).
    \end{equation*}
Note that $\p(X \in \partial (G^{\pm \varepsilon}) ) =0$ (the boundary is taken with respect to the uniform topology) due to the continuity of the distributions of $\sup_{0\leq t\leq 1} (X(t) - g^+(t))$ and $\inf_{0\leq t\leq 1}(X(t) - g^-(t))$ (see, e.g.,~p.\,232 in \cite{bill}). Furthermore, due to $X(1)$ having a continuous density it is clear that $\p( X(1) \in \partial B ) =0$ for any Borel set $B$ with $\partial B$ of Lebesgue measure zero. By subadditivity and the fact that $\partial (A_1\cap A_2) \subseteq \partial A_1 \cup \partial A_2 $ for arbitrary sets $A_1$ and $A_2$, it follows that
    \begin{equation*}
        \p(X \in \partial (G[B]) ) \leq \p(X \in \partial G) + \p(X(1) \in \partial B) =0.
    \end{equation*}
Hence one has from Corollary \ref{corr:gauss_bb_conv} that
    \begin{equation*}
        \liminf_{n\to \infty} \p(\Xbb \in G_n[B]) \geq \p(X \in G^{-\varepsilon}[B]), 
    \end{equation*}
and
    \begin{equation*}
         \limsup_{n\to \infty }\p(\Xbb \in G_n[B])\leq \p(X \in G^{+\varepsilon}[B]).
    \end{equation*}
As $\p(X \in \partial (G[B])) =0$ we also have
    \begin{equation*}
        \p(X \in G^{+\varepsilon}[B]) - \p(X \in G^{-\varepsilon}[B] ) = \p(X \in (\partial G)^{(\varepsilon)}[B]) \to 0\quad  \text{as } \varepsilon \downarrow 0,
    \end{equation*}
where $(\partial G)^{(\varepsilon)}$ is the $\varepsilon$-neighbourhood of $\partial G$ (also in the uniform norm). The result follows.

\end{proof}

\begin{proof}[Proof of relation \eqref{drop_C}]
Let $r := \norms{g^-}_{\infty} \vee \norms{g^+}_{\infty} $. Using the elementary inequality 
    \begin{equation*}
        \biggl\lvert 1 - \prod_{j=1}^n a_j\biggr\rvert \leq \biggl(\max_{i < n}\prod_{k=1}^i \lvert a_k \rvert\biggr) \sum_{j=1}^{n}\lvert 1 - a_j \rvert \leq \bigl(1 \vee \max_{k \leq n}\lvert a_k \rvert\bigr)^{n-1}\sum_{j=1}^n \lvert 1-a_j \rvert, 
    \end{equation*}
we have
    \begin{align*}
         &\lvert \p(\Xbb \in G_n) - \widehat{\mathbf{T}}_{n,1}\widehat{\mathbf{T}}_{n,2}\cdots \widehat{\mathbf{T}}_{n,n}\bm{1}^{\top}  \rvert\\
         &=\left| \sum_{\bm{x} \in E_n^G}\prod_{k=1}^n \frac{q_{n,k}(x_{k-1},x_k)}{C_{n,k}(x_{k-1})} -\sum_{\bm{x} \in E_n^G}\prod_{k=1}^n q_{n,k}(x_{k-1},x_k) \right|\\
         &=\left|\sum_{\bm{x} \in E_n^G}\prod_{k=1}^n \frac{q_{n,k}(x_{k-1},x_k)}{C_{n,k}(x_{k-1})} \left(1 -\prod_{i=1}^{n} C_{n,i}(x_{i-1}) \right)\right|\\
         &\leq \p(\Xbb \in G_n) \rho_n^{n-1}\sum_{i=1}^n \sup_{x \in E_{n,k}(r)}\lvert 1 - C_{n,i}(x) \rvert,
    \end{align*}
where $\rho_n := 1 \vee \max_{1\leq j\leq n}\sup_{\lvert x\rvert \leq r} C_{n,j}(x)$. Using Lemma \ref{lemma:C_sde}, we obtain \eqref{drop_C}.
\end{proof}

\section{Numerical results}

To illustrate the efficiency of our approximation scheme \eqref{eqn:bcp_gauss}, we implemented it in the programming language \texttt{Julia}. We used the package \texttt{HyperDualNumbers.jl} to evaluate the partial derivatives of $\driftX$ in \eqref{defn:b}.

It is well-known in the numerical analysis literature that trapezoidal quadrature is extremely accurate for analytic functions \cite{goodwin}. In light of relation \eqref{drop_C}, for numerical illustration purposes we drop the normalising constants $\C(x)$ and use $\widehat{\mathbf{T}}_{n,1}\widehat{\mathbf{T}}_{n,2}\cdots \widehat{\mathbf{T}}_{n,n}\bm{1}^{\top}$ instead of $\p(\Xbb \in G_n)$ to approximate $\p(X \in G)$.

\subsection{The Wiener process with one-sided boundary}

Using the method of images, the author in \cite{Daniels_1969} obtained a closed-form expression for the crossing probability of the boundary
\begin{equation*}
        g_D(t) := \half - t \log(\mbox{$\frac{1}{4}$}(1+  \sqrt{1 + 8e^{-1/t}})),\quad t>0,
    \end{equation*}
for the standard Wiener process $W:=\{W(t):t\geq 0\}$.

In order to make the state space $E_{n}^G$ finite in the case of the one-sided boundary where $g^-(t) = -\infty$, we insert an absorbing lower boundary at a low enough fixed level $L<x_0$ and replace $E_{n,k}^G$ with 
    \begin{gather*}
        E_{n,k}^{G,L} := \{x \in E_{n,k} : L < x < g_{n,k}^+ \},\quad k=1,\ldots,n,\\
        E_{n}^{G,L}:= E_{n,0}\times E_{n,1}^{G,L}\times \cdots E_{n,n}^{G,L}.
    \end{gather*}
We approximate $\p(W(t) < g_D(t),\, t\in [0,1])$ with $(\prod_{k=1}^n\mathbf{T}_{n,k}^L)\mathbf{1}^\top$, where $\mathbf{T}_{n,k}^L$ are substochastic matrices of dimensions $(\lvert E_{n,k-1}^{G,L} \rvert + 1 )\times (\lvert E_{n,k}^{G,L} \rvert  + 1)$ with entries equal to the respective transition probabilities
    \begin{align*}
        \begin{cases}
            q_{n,k}(x,y), & (x,y) \in E_{n,k-1}^{G,L} \times E_{n,k}^{G,L}, \\
            \sum_{z\in E_{n,k}\cap(-\infty,L]}q_{n,k}(x,z)& x \in E_{n,k-1}^{G,L},\, y = L,\\
            1, & x = L,\,y = L,\\
            0, & \text{otherwise},
        \end{cases}
    \end{align*}
where we put $f_n^-(t) = -\infty$ in the definition of $q_{n,k}$ in \eqref{defn:q}. This approximation assumes that the lower auxiliary boundary $L$ is sufficiently far away from the initial point $x_0$ and the upper boundary, such that after a sample path crosses the lower boundary it is highly unlikely that it will cross the upper boundary in the remaining time. In our example, we took $L=-3$. The probability of the Wiener process first hitting this level and then crossing $g_D$ prior to time $t=1$ is less than $1.26\times 10^{-6}$. Further, we chose $x_0 =0$, $\delta = 0$, $\gamma =2$ and $\minTimeSpan=\maxTimeSpan=1$. To guarantee convergence of the scheme, $\delta$ must be strictly positive, however for the values of $n$ we are interested in and the larger value of $\gamma$ compared to the one in the example from Figure \ref{figure_divergence}, the error is negligible.
\begin{figure}[ht]
    \centering
    \subfloat[The log-log plot of the absolute approximation error as a function of $n$. The crosses $\times$ show the absolute error of the Markov chain approximation {\em without} the Brownian bridge correction, while the bullets $\bullet$ show the error with the Brownian bridge correction. The upper and lower dashed lines correspond to $C_1n^{-1/2}$ and $C_2n^{-2}$ respectively, where $C_1$ and $C_2$ are fitted constants.   ]{\includegraphics{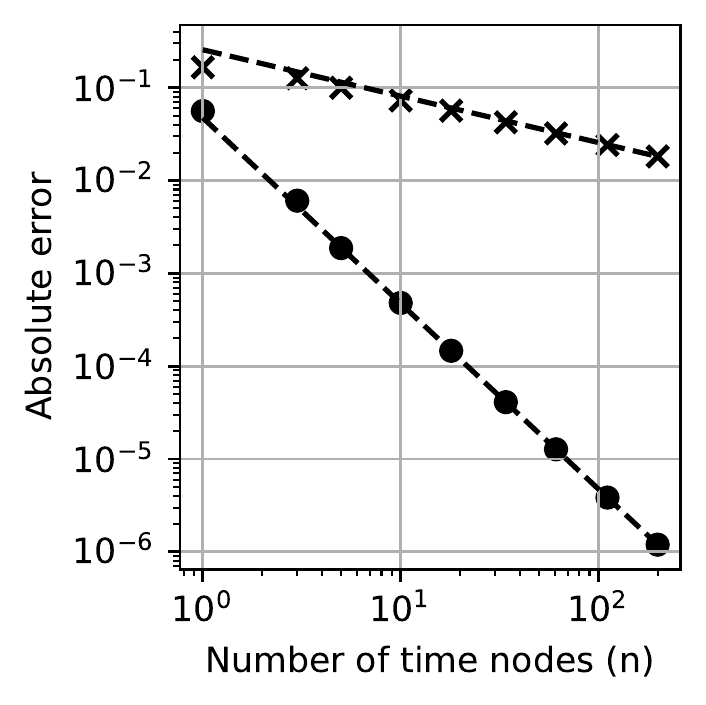} }%
    \quad
    \subfloat[Time evolution of the taboo transition density using the Markov chain approximation $\Xbb$ with $n=20$. The positions of the nodes on the surface correspond to the points from the respective $E_{n,k}^{G,-3}$. Note from \eqref{eqn:h} that the spacing between the nodes at the final time step $t =1$ is finer compared to earlier time steps. This is crucial for the observed improved convergence rate in (a).
    ]{\includegraphics{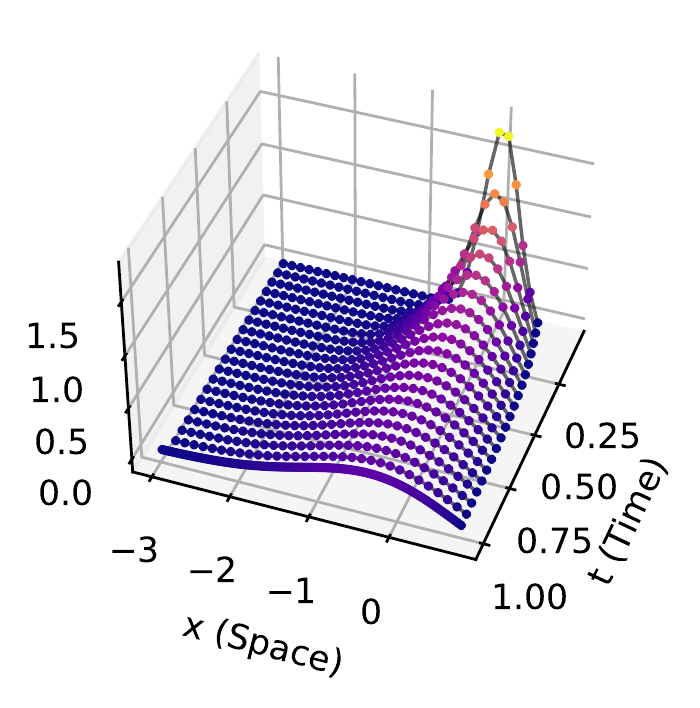}}%
    \caption{Approximation of the boundary $g_D$ non-crossing probabilities for the Wiener process. The exact Daniels boundary crossing probability in this case is $0.47974935\ldots$}%
    \label{fig:convergencePlotWiener}%
\end{figure}

From the log-log plot in Figure \ref{fig:convergencePlotWiener}, we see that the convergence rate $\lvert \p(W\in G) - \p(\Xbb \in G_n)\rvert$ is of the order $O(n^{-2})$, which is the same as the boundary approximation order of the error $\lvert \p(W \in G) - \p(W \in G_n)\rvert = O(n^{-2})$ proved in~\cite{borov} in the case of twice continuously differentiable boundaries. It appears that, due to the high accuracy of approximation of the increments' moments (Lemma \ref{lemma:moment_gauss_sde}) and the Brownian bridge correction applied to the transition probabilities, we achieve a much faster convergence rate compared to the convergence rate $O(n^{-1/2})$ achieved in \cite{fuwu} (as shown by the $\times$ crosses in Figure 1\,(a)).

\subsection{The Ornstein--Uhlenbeck process}

Let $X$ be the Ornstein--Uhlenbeck (OU) process satisfying the following stochastic differential equation:
    \begin{equation*}
        \begin{cases}
            dX(t) = -X(t)\, dt +  dW(t),\quad  t \in (0,1],\\
            X(0) = 0.
        \end{cases} 
    \end{equation*}
The usual approach for computing the boundary crossing probability of the OU process is to express the process in terms of a time-changed Brownian motion. This is achieved by using the time substitution $\theta(t) := (e^{2 t}-1)/2$, so that we can write $X(t) = e^{-t}W(\theta(t))$. 

To illustrate the effectiveness of our approximation, we consider the following two-sided boundary for which explicit boundary crossing probabilities are available for the OU process:
    \begin{align*}
    g_{\psi}^{\pm}(t) &:= e^{-t}\psi_{\pm}(\theta(t)) ,\, \mbox{where } \psi_{\pm}(t) = \pm \half t \cosh^{-1}(e^{4/t}),\quad t>0.
    \end{align*}
Letting $t := \theta(s)$, we obtain
    \begin{align*}
        P_{\psi}(T) &:= \p(\psi_{-}(t) < W(t) < \psi_{+}(t), \,0 \leq t \leq T)\\ 
        &= \p(e^{-s}\psi_{-}(\theta(s)) < e^{- s}W(\theta(s)) < e^{- s}\psi_{\pm}(\theta(s)),\, 0 \leq s \leq \theta^{-1}(T))\\
        &= \p(g_{\psi}^{-}(s) < X(s) < g_{\psi}^{+}(s),\, s \in [0,\theta^{-1}(T)]),
    \end{align*}
where we set $T := \theta(1)$ so that $s\in[0,1]$. A closed-form expression for $P_{\psi}(T)$ can be found on p.\,28 in \cite{lerche}. The exact boundary crossing probability in this case is $0.75050288\ldots.$

Using \eqref{defn:b} and \eqref{defn:sigma}, the approximate drift and diffusion coefficients of the weak second-order It\^o--Taylor expansion for the OU process are given by
    \begin{equation*}
        \driftZeta_{n,k}(x) = - x + \half\Delta t_{n,k} x , \quad \stddevZeta_{n,k}(x) = 1 - \half\Delta t_{n,k}.
    \end{equation*}
From the numerical results below, it appears that it is sufficient to use the weak second-order It\^o--Taylor expansion of transition densities instead of the true transition density of the OU process for our Markov chain approximation to maintain a $O(n^{-2})$ convergence rate of the boundary crossing probabilities.      
\begin{figure}[ht]
    \centering
    \subfloat[The log-log plot of the absolute approximation error as a function of $n$. The upper and lower dashed lines correspond to $C_1n^{-1/2}$, $C_2n^{-1}$ and $C_3n^{-2}$ respectively, where $C_1$, $C_2$ and $C_3$ are fitted constants. The crosses $\times$ show the absolute error of the Markov chain approximation {\em without} the Brownian bridge correction, while the bullets $\bullet$ show the error with the Brownian bridge correction. The markers $\bigtriangleup$ and $\circ$ show the absolute error of the Markov chain approximation using the exact transition density of the OU process and the Euler--Maruyama approximation instead of the transition density from the It\^o--Taylor expansion, respectively.]{\includegraphics{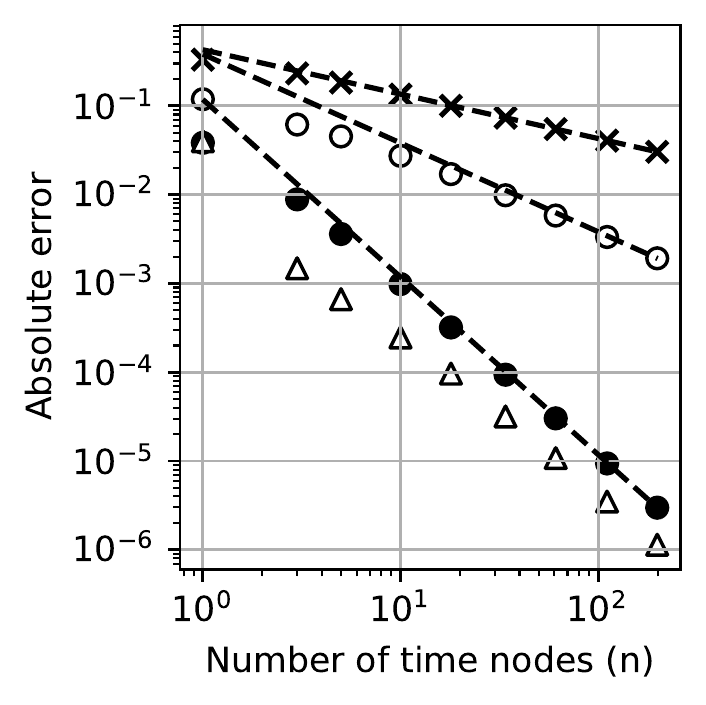} }%
    \quad
    \subfloat[Time evolution of the taboo transition density using the Markov chain approximation $\Xbb$ with $n=20$. The positions of the nodes on the surface correspond to the points from the respective $E_{n,k}^{G}$. Note from \eqref{eqn:h} that the spacing between the nodes at the final time step $t =1$ is finer compared to earlier time steps. This is crucial for the observed improved convergence rate in (a).
    ]{\includegraphics{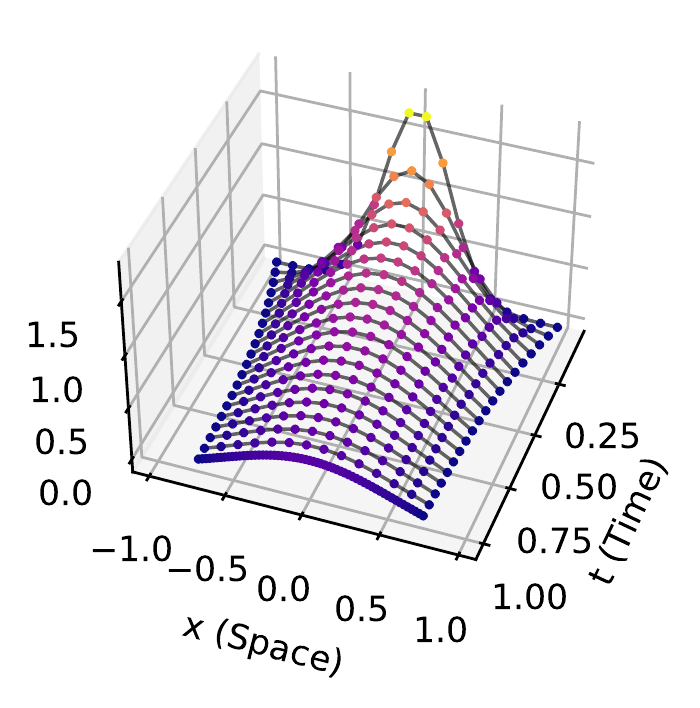}}%
    \caption{Approximation of non-crossing probabilities of the OU process with boundaries $g_{\psi}^{\pm}$.}%
    \label{fig:convergencePlotOU}%
\end{figure}
Using the log-log plot in Figure \ref{fig:convergencePlotOU}, we empirically observe that the convergence rate $\lvert \p(X\in G) - \p(\Xbb \in G_n)\rvert$ is of the order $O(n^{-2})$, which is the same as the boundary approximation order of error $\lvert \p(X \in G) - \p(X \in G_n)\rvert = O(n^{-2})$ proved in \cite{Downes2008} in the case of twice continuously differentiable boundaries. It appears that, at this level of accuracy, one might ignore the higher order terms in the diffusion bridge crossing probability derived in~\cite{baldi02}.

\end{document}